\documentclass{amsart}
\usepackage{mathrsfs}
\usepackage{epic,eepic,epsf,epsfig,tikz}
\usetikzlibrary{positioning}
\usepackage{amsfonts,srcltx,mathrsfs}
\usepackage{chngcntr}
\counterwithin{table}{section}
\counterwithin{figure}{section}
\usepackage[margin=1.5cm]{geometry}
\usepackage{multirow}
\usepackage{amsmath}
\usepackage{amssymb}
\usepackage{amsbsy}
\usepackage{graphicx}
\usepackage{amsfonts}
\usepackage{color}
\newtheorem{lem}{Lemma}[section]%
\newtheorem{theorem}[lem]{Theorem}%
\newtheorem{defi}[lem]{Definition}%
\newtheorem{cor}[lem]{Corollary}%
\newtheorem{prop}[lem]{Proposition}%
\newtheorem{notation}[lem]{Notation}

\def\a{\alpha} \def\b{\beta}

 \def\O{\Omega} \def\G{\Gamma}

 \def\lg{\langle} \def\rg{\rangle}
\def\nd{\mathrel{\bigm|\kern-.7em/}}

\def\f{\noindent}

\def\PSL{\hbox{\rm PSL}}
   
\def\P\GammaL{\hbox{\rm P\Gamma L}}

\def\Aut{\hbox{\rm Aut}}

\def\Cay{\hbox{\rm Cay}}
\def\BiCay{\hbox{\rm BiCay}}

\def\mz{{\mathbb Z}}

\begin{document}
\title[$m$-graphical regular representation of groups]{A classification of the $m$-graphical regular representation of finite groups}

\footnotetext[1]{Corresponding authors. E-mails: JiaLiDu$@$bjtu.edu.cn,
yqfeng$@$bjtu.edu.cn, pablo.spiga$@$unimib.it}

\author{Jia-Li Du}
\address{Jia-Li Du, Department of Mathematics, Beijing Jiaotong University, Beijing 100044, China}
\email{JiaLiDu@bjtu.edu.cn}

\author{Yan-Quan Feng}
\address{Yan-Quan Feng, Department of Mathematics, Beijing Jiaotong University, Beijing 100044, China}
\email{yqfeng@bjtu.edu.cn}

\author{Pablo Spiga}
\address{Pablo Spiga, Dipartimento di Matematica e Applicazioni, University of Milano-Bicocca, Via Cozzi 55, 20125 Milano, Italy}
\email{pablo.spiga@unimib.it}

\begin{abstract}
In this paper we extend the classical notion of \textit{digraphical} and \textit{graphical regular representation} of a group and we classify, by means of an explicit description, the finite groups satisfying this generalization. A graph or digraph is called {\em regular} if each vertex has the same valency, or, the same out-valency and the same in-valency, respectively. An {\em $m$-(di)graphical regular representation} (respectively, $m$-GRR and $m$-DRR, for short) of a group $G$ is a regular (di)graph whose automorphism group is isomorphic to $G$ and acts semiregularly on the vertex set with $m$ orbits. When $m=1$, this definition agrees with the classical notion of GRR and DRR. Finite groups admitting a $1$-DRR were classified by Babai in 1980, and the analogue classification of finite groups admitting a $1$-GRR was completed by Godsil in 1981. Pivoting on these two results in this paper we classify finite groups admitting an $m$-GRR or an $m$-DRR, for arbitrary positive integers $m$. For instance, we prove that every non-identity finite group admits an $m$-GRR, for every $m\ge 5$.

\smallskip

\end{abstract}


 \maketitle
\section{Introduction}

By a {\em digraph} $\G$, we mean an ordered pair $(V,A)$ where the \emph{vertex set} $V$ is a non-empty set and the {\em arc set} $A\subseteq V\times V$ is a binary relation on $V$. The elements of $V$ and $A$ are called \emph{vertices} and \emph{arcs} of $\G$, respectively. An \emph{automorphism} of $\G$ is a permutation $\sigma$ of $V$ that preserves the relation $A$, that is, $(x^\sigma,y^\sigma)\in A$ for every $(x,y)\in A$. Throughout this paper, all groups and digraphs are finite,
and all digraphs are regular, that is, there exists an integer $d$ such that, for every vertex $v$, the in-valency and the out-valency of $v$ are both equal to $d$.  Moreover, our digraphs have no loops or multiple arcs. The digraph $\Gamma=(V,A)$ is a graph if the binary relation $A$ is symmetric.

Let $G$ be a permutation group on a set $\O$ and let $\a\in \O$.
Denote by $G_{\a}$ the stabilizer of $\a$ in $G$, that is,
the subgroup of $G$ fixing $\a$. We say that
$G$ is {\em semiregular} on $\O$ if $G_\a=1$ for every
$\a \in \O$, and {\em regular} if it is semiregular and transitive.

An {\em $m$-Cayley (di)graph} $\G$
over a group $G$ is defined as a (di)graph which has
a semiregular group of automorphisms isomorphic to $G$
with $m$ orbits on its vertex set. When $m=1$, $1$-Cayley (di)graphs are the usual Cayley (di)graphs. Moreover, when $m=2$, $2$-Cayley (di)graphs are also called {\em bi-Cayley} (di)graphs in the literature.

We say that a group $G$ admits an {\em $m$-(di)graphical
regular representation} (respectively $m$-GRR and $m$-DRR, for short)
if there exists a {\em regular} $m$-Cayley digraph $\G$ over
$G$ such that $\Aut(\G)\cong G$. In particular, $1$-GRRs and $1$-DRRs are the usual GRRs and DRRs, and $2$-GRRs and $2$-DRRs are also called  Bi-GRRs and Bi-DRRs in the literature.

When studying a Cayley digraph over a finite group $G$, a very important question is to
determine whether $G$ is in fact the full automorphism group. For this reason, DRRs and GRRs have been widely studied. The most natural question is the
``GRR and DRR problem": which groups admit GRRs and DRRs~?

Babai~\cite{Babai} proved that every group admits a DRR except for  $Q_8$, $\mz_2^2$,
$\mz_2^3$, $\mz_2^4$ and $\mz_3^2$. It is clear that if
a group admits a GRR then it also admits a DRR, but the converse is not true. GRRs turned out to be much more difficult to handle and, after a long series of partial results by various authors~\cite{Hetzel,Imrich,ImrichWatkins,ImrichWatkins2,NowitzWatkins1,NowitzWatkins2,Watkins}, the classification was completed by Godsil in~\cite{Godsil}. In this area, the work of Imrich and Watkins turned out to be very influential.

In the literature, there are several generalizations for GRRs and DRRs and, for more results, we refer to \cite{BabaiI,Godsil2, Hujdurovic,KMMS,MorrisSpiga,MorrisSpiga2,MorrisSpiga1,MorrisTymburski,Spiga,XiaF}.
In this paper, we are concerned with $m$-GRRs and $m$-DRRs; we prove the following results:

\begin{theorem}\label{theo=main}
Let $m$ be a positive integer and let $G$ be a finite group. Then either $G$ admits an $m$-$\mathrm{GRR}$ or $(m,G)$ is in Table~$\ref{table1}$. Conversely, if $(m,G)$ is in Table~$\ref{table1}$, then $G$ has no $m$-$\mathrm{GRR}$.
\end{theorem}

\begin{table}[!ht]
\begin{tabular}{|c|c|c|}\hline
$m$&Group&Comments\\\hline
$1$&Abelian groups of exponent greater than $2$&\\
&Generalized dicyclic groups&\\
&$\mz_2^\ell$&$\ell\in \{2,3,4\}$\\
&$D_n$ &$n\in \{6,8,10\}$\\
&$Q_8\times \mz_3$, $Q_8\times \mz_4$&\\
&$\mathrm{Alt}(4)$&\\
&$\langle a,b,c~|~a^2=b^2=c^2=1,abc=bca=cab\rangle$&order $16$\\
&$\langle a,b~|~a^8=b^2=1,bab=a^5\rangle$& order $16$\\
&$\langle a,b,c~|~a^3=b^3=c^2=1,ab=ba,(ac)^2=(cb)^2=1\rangle$&order $18$\\
&$\langle a,b,c~|~a^3=b^3=c^3=1,ac=ca,bc=cb,b^{-1}ab=ac\rangle$& order $27$\\\hline
$2$&$Q_8$, $\mz_2^2$, $\mz_n$&$n\in \{1,2,3,4,5\}$\\\hline
$3$&$\mz_n$&$n\in \{1,2,3\}$\\\hline
$4$&$\mz_n$&$n\in \{1,2\}$\\\hline
$5\le m\le 9$&$\mz_1$&\\\hline
\end{tabular}
\vskip0.5cm
\caption{Groups not admitting an $m$-GRR}\label{table1}
\end{table}

\begin{theorem}\label{theo=Bi-DRR}
Let $m$ be a positive integer and let $G$ be a finite group. One of the following holds:
\begin{enumerate}
\item $G$ has an $m$-$\mathrm{DRR}$;
\item $m=1$ and $G$ is isomorphic to $Q_8$, $\mathbb{Z}_2^2$, $\mathbb{Z}_2^3$, $\mathbb{Z}_2^4$ or $\mathbb{Z}_3^2$;
\item $m=2$ and $G$ is isomorphic to either $\mz_1$ or $\mz_2$;
\item $3\leq m\leq 5$ and $G$ is isomorphic to $\mz_1$.
\end{enumerate}
\end{theorem}

In Section~\ref{sec2}, we give some notation and some preliminaries, used throughout the paper. In Section~\ref{sec3}, we prove Theorem~\ref{theo=main} for non-abelian groups $G$ admitting a GRR. Then, in Section~\ref{sec4}, we prove Theorem~\ref{theo=main} for abelian groups and for groups not admitting a GRR. Theorem~\ref{theo=Bi-DRR} follows immediately from Theorem~\ref{theo=main} and we prove it in Section~\ref{sec5}. In what follows we heavily rely to some computer computations for dealing with small groups and we acknowledge the invaluable help of the computer algebra system \texttt{magma}~\cite{magma}. Most of our work is devoted to $2$-Cayley graphs and we refer the reader to~\cite{AraluzeKKMM,KovacsKMW,KovacsMMM,MalnicMS,ZF,ZF1} for some background work on this family of graphs.

\section{Notation and Preliminaries}
\label{sec2}

\subsection{Notation}\label{notation}

Given a graph $\Gamma$ and a subset $X$ of the vertex set $V\Gamma$ of $\Gamma$, we denote by $\Gamma[X]$ the {\em subgraph induced} by $\Gamma$ on $X$. When the graph $\Gamma$ is clear from the context or when the subset $X$ is clearly referring to a particular graph $\Gamma$, we simply write $[X]$ for $\Gamma[X]$, and this should cause no confusion.

Let $\G$ be a graph, let $v$ be a vertex of $\G$ and let $i\in \mathbb{N}$, the set of non-negative integers. We let $\G_i(v)$ denote the vertices of $\G$
having {\em distance} $i$ from $v$, consistently with the notation above, we let $\G[\G_i(v)]$ (or simply $[\G_i(v)]$) denote the
subgraph of $\G$ induced on $\G_i(v)$. Observe that $\G_0(v)=\{v\}$ and $\G_1(v)=\G(v)$ is the {\em neighborhood} of the vertex $v$ in $\G$.

Given a group $G$ and $g\in G$, we let $o(g)$ denote the {\em order} of the element $g$ and we let $\mathbf{Z}(G)$ denote the {\em center} of $G$. Given a positive integer $m$, consistently throughout the whole paper, for not making our notation too cumbersome to use, we denote the element $(g,i)$ of the {\em Cartesian product} $G\times \{0,\ldots,m-1\}$ simply by $g_i$. A subset $R$ of $G$ is said to be a {\em Cayley subset} if $R=R^{-1}:=\{r^{-1}\mid r\in R\}$ and $1\notin R$.

Let $\G$ be a Cayley digraph over a group $G$. It is well-known that $G$ has a subset $R$ with $1\not\in R$ such that $\G\cong \Cay(G,R)$, where $\Cay(G,R)$ is the digraph with vertex set $G$ and arc set $\{(g,rg) |\ g\in G,r\in R\}$. It is easy to see that $\Cay(G,R)$ is connected if and only if $R$ generates the group $G$, and $\Cay(G,R)$ is undirected if and only if $R$ is a Cayley subset of $G$.
Furthermore, the group consisting of the permutations $x\mapsto xg$ on $G$, for all $g\in G$,
is a regular subgroup of $\Aut(\Cay(G, S))$ and, by abuse of notation, we let  $G$ denote this regular subgroup.

Let $\G$ be an $m$-Cayley digraph over a group $G$. Similar to Cayley digraphs, $G$ has subsets $T^{i,j}$ with $1\notin T^{i,i}$ for $i,j\in\{0,\ldots,m-1\}$, and $\G$ is isomorphic to the digraph with vertex set
$G\times \{0,\ldots,m-1\}$ and arc set  $$\bigcup_{i,j} \{(g_i, (tg)_j)~|~t\in T^{i,j}\},$$ which has a semiregular group of automorphisms consisting of right multiplications by elements of $G$. Again, by abuse of notation, we let $G$ denote this semiregular group. It is also easy to see that this digraph is undirected if
and only if $(T^{j,i})^{-1}=$ $T^{i,j}$ for all $i,j\in \{0,\ldots,m-1\}$. In particular, we let $$\BiCay(G,R,L,S)$$
denote  the $2$-Cayley graph (that is, bi-Cayley digraph) with $R=T^{0,0}$, $L=T^{1,1}$ and $S=T^{0,1}$ (where we are implicitly assuming $T^{1,0}=(T^{0,1})^{-1}$).

Given a graph $\Gamma$, we denote by $\Gamma^c$ the {\em complement} of $\Gamma$. In particular, when $\G=\Cay(G,R)$ for some Cayley subset $R$ of $G$, we have $\Gamma^c=\Cay(G,R^c)$ where we define $R^c:=G\setminus (\{1\}\cup R)$.

Given a graph $\G$ and a subset $X$ of the vertex set $V\G$, we denote by $\Gamma[X]_I$ the set of
{\em isolated vertices} of $\G[X]$, that is, all vertices $v\in X$ with $X\cap \G(v)=\emptyset$. For instance, when $\G=\Cay(G,R)$, $\G[R]_I$ denotes the set of isolated vertices of the subgraph induced by $\G$ on the neighborhood $R$ of the vertex $1$. Similarly, $\Gamma^c[R^c]_I$ is the set of isolated vertices of the subgraph induced by $\G^c$ on the neighborhood $R^c$ of the vertex $1$.

Given two graphs $\G_1$ and $\G_2$, we denote by $\G_1\uplus\G_2$ the {\em disjoint union} of $\G_1$ and $\G_2$. Given $n\in\mathbb{N}$ with $n\ge 1$, we denote by $\mathbf{K}_n$ the {\em complete graph} on $n$ vertices.

Given a graph $\G$ and a subgroup $G$ of the automorphism group $\Aut(\G)$ of $\G$, we say that $\G$ is $G$-{\em vertex-transitive} if $G$ acts transitively on the vertex set of $\G$. Moreover, given a vertex $v$ of $\G$, we denote by $G_v$ the {\em vertex stabilizer} of $v$ in $G$, that is, $G_v:=\{g\in G\mid v^g=v\}$.

We let $Q_8$ denote the {\em quaternion group} of order $8$, we let $D_n$
denote the {\em dihedral group} of order $n$, we let $\mathbb{Z}_n$  denote the {\em cyclic group} of order $n$ and we let $\mathrm{Alt}(n)$ denote the {\em alternating group} of degree $n$.

\subsection{Preliminaries}\label{sec=GRR}
Let $A$ be an abelian group of even order and of exponent greater than $2$, and let $y$ be an
involution of $A$. The {\em generalized dicyclic} group $\mathrm{Dic}( A , y , x )$ is the group
$\langle A , x \mid x^2 = y, a^x = a^{-1}, \forall a \in A \rangle$.
A group is called generalized dicyclic if it is isomorphic to some $\mathrm{Dic}( A , y , x )$. When $A$ is cyclic, $\mathrm{Dic}( A , y , x )$
is called {\em dicyclic} or {\em generalized quaternion}.

The following proposition, extracted from the work of Godsil~\cite{Godsil}, gives the list of groups admitting no GRR.

\begin{prop}\label{prop=GRR}{\rm \cite[Theorem 1.2]{Godsil}}
A finite group $G$ admits a $\mathrm{GRR}$ unless $G$ belongs to one of
the following classes:
\begin{itemize}
  \item the abelian groups of exponent greater than two;
  \item the generalized dicyclic groups;
  \item the following thirteen ``exceptional groups":
\begin{align*}
&\mz_2^2,\,\,\, \mz_2^3,\,\,\, \mz_2^4,\,\,\,
D_6,\,\,\, D_8,\,\,\, D_{10},\,\,\,
Q_8\times \mz_3,\,\,\, Q_8\times \mz_4,\,\,\,
\mathrm{Alt}(4),\\
&\langle a,b,c~|~a^2=b^2=c^2=1,abc=bca=cab\rangle\quad (\textrm{of order }16),\\
&\langle a,b~|~a^8=b^2=1,bab=a^5\rangle\quad (\textrm{of order }16),\\
&\langle a,b,c~|~a^3=b^3=c^2=1,ab=ba,(ac)^2=(cb)^2=1\rangle\quad(\textrm{of order }18),\\
&\langle a,b,c~|~a^3=b^3=c^3=1,ac=ca,bc=cb,b^{-1}ab=ac\rangle\quad(\textrm{of order }27).
\end{align*}
\end{itemize}
\end{prop}

In the following, we give a basic lemma, which will be used in Section~\ref{sec3}.

\begin{lem}\label{lem=R_I}
Let $\G:=\Cay(G,R)$ be a $\mathrm{GRR}$. Then, either $|\G[R]_I|\leq 1$
or $|\G^c[R^c]_I|\leq 1$. Moreover, when $|G|>2$, either $G=\langle R\setminus \G[R]_I\rangle$ and $|\G[R]_I|\le 1$, or
$G=\langle R^c\setminus \G^c[R^c]_I\rangle$ and $|\G^c[R^c]_I|\le 1$.
\end{lem}

\begin{proof} By abuse of notation, we write $R_I:=\G[R]_I$ and $R_I^c:=\G^c[R^c]_I$ for simplicity. If $|G|\le 2$, then the lemma is straightforward. Therefore, we may assume $|G|>2$; in particular, by Proposition~\ref{prop=GRR}, groups $G$ with $3\le |G|\le 11$ do not admit GRRs. Therefore $|G|\geq 12$
and $\G$ is connected. Observe that, since $\G$ is not a cycle, the Cayley subset of a GRR for $G$ has cardinality at least $3$.

\smallskip

\noindent\textsc{Claim:} Let $S$ be a Cayley subset of $G$ with $|S|\geq|G|/2$ and let $\Sigma:=\Cay(G,S)$. If $|S|>|G|/2$, then $|\Sigma[S]_I|=0$, and if $|S|=|G|/2$,
then $|\Sigma[S]_I|\le 1$ and $G=\langle S\setminus \Sigma[S]_I\rangle$.

\smallskip
Since $\Sigma$ is a GRR, $\Sigma$ is connected and hence $G=\langle S\rangle$. First assume that $|S|>|G|/2$.
We argue by contradiction and suppose that $\Sigma[S]_I\ne\emptyset$.
Given $x\in \Sigma[S]_I$, we have $\Sigma(x)\subseteq S^c\cup\{1\}$ and hence
$|S|=|\Sigma(x)|\le |S^c|+1=(|G|-|S|-1)+1=|G|-|S|$. Thus $|S|\le |G|/2$, contrary to our assumption.

Now assume $|S|=|G|/2$. Again we argue by contradiction and we suppose that $|\Sigma[S]_I|\ge 2$. Let $x,y\in \Sigma[S]_I$.
Then $\Sigma(x),\Sigma(y)\subseteq S^c\cup\{1\}$. Since $|\Sigma(x)|=|\Sigma(y)|=|G|/2$
and $|S^c|+1=(|G|-|S|-1)+1=|G|/2$, we deduce $\Sigma(x)=\Sigma(y)$. Therefore,
the transposition $(x\,y)$ on $V\Sigma$ (interchanging
$x$ and $y$ and fixing all other vertices) is an automorphism of $\Sigma$,
contradicting that $\Sigma$ is a GRR. Thus $|\Sigma[S]_I|\le 1$.

If $\Sigma[S]_I=\emptyset$, then $G=\langle S\rangle=\langle S\setminus \Sigma[S]_I\rangle$.
Suppose $|\Sigma[S]_I|=1$ and $\Sigma[S]_I:=\{x\}$. If $\langle S\setminus \Sigma[S]_I\rangle\ne G$,
then $\langle S\setminus \Sigma[S]_I\rangle$ is a proper subgroup of $G$.
Thus $|G|/2=|S|=|S\setminus \Sigma[S]_I|+1\le |\langle S\setminus
\Sigma[S]_I\rangle|\le |G|/2$. This shows that $H:=(S\setminus\{x\})\cup\{1\}$
is a subgroup of $G$ having index $2$ and $S=(H\setminus\{1\})\cup\{x\}$.
In particular, $\Sigma[H]$ and $\Sigma[G\setminus H]$ are both complete
graphs and $\Sigma$ has a matching between $H$ and $G\setminus H$.
However, from this explicit description of $\Sigma$, we obtain a
contradiction because $\Sigma$ is a GRR.~$_\blacksquare$

\smallskip

As $\G$ and its complement $\G^c$ have the same
automorphism group and as $\G$ is a GRR over $G$,
$\G^c$ is also a GRR over $G$. Moreover, since
$|G|-1=|R|+|R^c|$, replacing $\G$ by $\G^c$ if necessary,
we may assume $3\le |R^c|\leq |R|$. Moreover, from Claim,
we may assume $|R|<|G|/2$ and $|R^c|<|G|/2$. In particular,
 since $|G|-1=|R|+|R^c|$, we deduce $|R|=|R^c|=(|G|-1)/2$ and $G$ has odd order.

We now argue by contradiction and we suppose $|R_I|\geq 2$
and $|R^c_I|\geq 2$. Let $x,y\in R_I$ with $x\neq y$.
Then $\G(x)\subseteq R^c\cup \{1\}$. If $\G(x)=\G(y)$,
then the transposition $(x\,y)$ on $V(\G)$ (interchanging
$x$ and $y$ and fixing all other vertices) is an automorphism
of $\G$, contradicting that $\G$ is a GRR. Thus $\G(x)\not=\G(y)$.

Since $1\in \G(x)$ and $1\in\G(y)$, there exist $u,v\in R^c$
such that
\begin{align}\label{eq1}
\G(x)=\{1\}\cup(R^c\setminus\{u\}),\quad\G(y)=\{1\}\cup(R^c\setminus\{v\}).
\end{align}
Moreover, as $\G(x)\not=\G(y)$, we have $u\neq v$, see Figure~\ref{Fignine}.

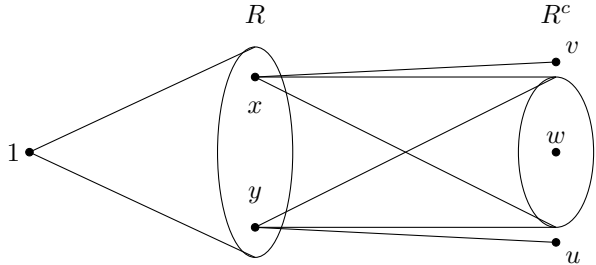
\begin{figure}
\begin{tikzpicture}
\draw (7,0) ellipse (0.5 and 1);
\draw[fill] (0,0) circle [radius=0.05];
\node [left] at (0,0){$1$};
\draw[fill] (3,1) circle [radius=0.05];
\node [below] at (3,.8){$x$};
\draw[fill] (3,-1) circle [radius=0.05];
\node [above] at (3,-.8){$y$};
\draw[fill] (7,-1.2) circle [radius=0.05];
\node [right] at (7,-1.4){$u$};
\draw[fill] (7,1.2) circle [radius=0.05];
\node [right] at (7,1.4){$v$};
\draw[fill] (7,0) circle [radius=0.05];
\node [above] at (7,0){$w$};
\node [above] at (3,1.6){$R$};
\node [above] at (7,1.6){$R^c$};
\draw (0,0)--(3,1.4);
\draw (0,0)--(3,-1.4);
\draw (3,1)--(7,1.2);
\draw (3,-1)--(7,-1.2);
\draw (3,1)--(7,1);
\draw (3,1)--(7,-1);
\draw (3,-1)--(7,1);
\draw (3,-1)--(7,-1);
\draw (3,0) ellipse (0.5 and 1.4);
\end{tikzpicture}
\caption{Figure for the proof of Lemma~\ref{lem=R_I}}\label{Fignine}
\end{figure}

We show that, for every $w\in R^c\setminus\{u,v\}$, $w$ is adjacent
to both $x$ and $y$. (Recall $|R^c|\geq 3$.) By~\eqref{eq1},
$x\in \G(w)$ and $y\in\G(w)$ and, since $G=R\cup R^c\cup \{1\}$
and $|\G(w)|=|R|=|R^c|$, we infer that $R^c\setminus\{w\}$ contains
a vertex $w'$ that is not adjacent to $w$ in $\G$. Therefore,
$w$ and $w'$ are adjacent in $\G^c$. Thus $w\not\in R^c_I$.
However, since $|R^c_I|\geq 2$ and since $w$ is an arbitrary
element in $R^c\setminus\{x,y\}$, we have $R^c_I=\{u,v\}$.
From this, it follows that $u$ and $v$ are adjacent to $w$.
Since this argument does not depend upon $w\in R^c\setminus \{u,v\}$,
we have
\begin{align}\label{eq2}
\G(u)= \{y\}\cup(R^c\setminus\{u\}),\quad
 \G(v)=\{x\}\cup(R^c\setminus\{v\}).
\end{align}
By~\eqref{eq1} and~\eqref{eq2}, we infer that the permutation $(x\,y)(u\,v)$
on $V\G$ is an automorphism of $\G$, contradicting that $\G$
is a GRR. Thus, we have shown that either $|R_I|\le 1$ or $|R^c_I|\le 1$. Replacing $R$ by $R^c$ if necessary, we may assume that $|R_I|\le 1$.

It remains to show that $G=\langle R\setminus R_I\rangle$. If $R_I=\emptyset$, then $\langle R\setminus R_I\rangle=\langle R\rangle=G$ because $\G$ is a GRR.  For the rest of the proof, we may assume $|R_I|=1$. Write $R_I:=\{x\}$.  If $\langle R\setminus R_I\rangle\ne G$, then $\langle R\setminus R_I\rangle$ is a proper subgroup of $G$. Since $G$ has odd order, we have $|R|=|R\setminus R_I|+1\le |\langle R\setminus R_I\rangle|\le |G|/3=(1+2|R|)/3$ and hence $|R|\le 1$. In particular, $|G|\le 3$, contradicting $|G|\ge 12$.
\end{proof}

\begin{lem}\label{lem=1234}
Let $G$ be a non-abelian group. Then there exists $x\in G$ with $x\notin\mathbf{Z}(G)$ and $o(x)>2$.
\end{lem}
\begin{proof}
We argue by contradiction and we suppose that for every $x\in G$ either $x\in\mathbf{Z}(G)$ or $o(x)=2$. Since $G$ is non-abelian, there exist $x,y\in G$ with $xy\ne yx$. In particular, since $x,y\notin \mathbf{Z}(G)$, we deduce $o(x)=o(y)=2$. If $o(xy)=2$, then $xy=(xy)^{-1}=y^{-1}x^{-1}=yx$, a contradiction. Therefore, $o(xy)>2$ and hence $xy\in \mathbf{Z}(G)$. Thus $xy=(xy)^x=x(xy)x=x^2yx=yx$, a contradiction.
\end{proof}

\section{Non-abelian groups admitting GRRs}
\label{sec3}

In this section, we prove Theorem~\ref{theo=main} for non-abelian groups admitting GRRs.

\begin{defi}\label{defi:22}
{\rm Let $G$ be a non-abelian group, let $m$ be a positive integer with $m\ge 2$, let $R$ be a Cayley subset of $G$ with $\G:=\Cay(G,R)$ a GRR, $|\G[R]_I|\le 1$ and $G=\langle R\setminus \G[R]_I\rangle$, and let $x\in G$. Suppose that one of the following holds:
\begin{enumerate}
\item\label{part1} $m\ge 3$, $x\in G\setminus\mathbf{Z}(G)$ and $o(x)>2$, or
\item\label{part2} $m=2$ and $x^2\in G\setminus\mathbf{Z}(G)$, or
\item\label{part3} $m=2$, $\{g^2\mid g\in G\}\subseteq \mathbf{Z}(G)$, $x\in G\setminus(\mathbf{Z}(G)\cup R)$ and $o(x)>2$.
\end{enumerate}
When~\eqref{part1} or~\eqref{part3} holds,  we let $\Theta^m(G,R,x)$ be the $m$-Cayley graph with vertex set $G\times \{0,\ldots,m-1\}$ and with edge set consisting of the pairs:
\begin{description}
\item[(i)] $\{g_i,(rg)_i\}$, for every $i\in \{0,\ldots,m-1\}$, $g\in G$ and $r\in R$;
\item[(ii)] $\{g_i,g_{i+1}\}$, for every $g\in G$ and $i\in \{0,1,\ldots,m-2\}$;
\item[(iii)] $\{g_0,(xg)_{m-1}\}$, for every $g\in G$.
\end{description}
Observe that $\Theta^m(G,R,x)$ is regular of valency $|R|+2$.

When~\eqref{part2} holds, we let $\Theta^2(G,R,x)$ be the $2$-Cayley graph with vertex set $G\times \{0,1\}$ and with edge set consisting of the pairs $\{\{g_0,(rg)_0\},\{g_0,(rg)_1\},\{g_0,(xg)_1\}\mid g\in G\}$. Observe that $\Theta^2(G,R,x)$ is regular of valency $|R|+1$.

The three cases~\eqref{part1},~\eqref{part2} and~\eqref{part3} are mutually disjoint.
The existence of a Cayley subset $R$ satisfying $|\G[R]_I|\le 1$ and $G=\langle R\setminus \G[R]_I\rangle$ is guaranteed by Lemma~\ref{lem=R_I}. Moreover, the existence of $x\in G$ satisfying $x\in G\setminus\mathbf{Z}(G)$ and $o(x)>2$ is guaranteed by Lemma~\ref{lem=1234}. In particular, when $m\ge 3$,~\eqref{part1} is always satisfied for some $x$. The case $m=2$ is more problematic, because~\eqref{part2} and \eqref{part3} do not cover all possibilities. (For instance, when $G=\mathbf{Z}(G)\cup R\cup\{g\in G\mid o(g)=2\}$ we do not have any choice for $x$.)

Figure~\ref{FigFigFigFig} might be of some help for familiarizing with the structure of $\Theta^m(G,R,x)$.
\begin{figure}[!hhh]
\begin{center}
\begin{tikzpicture}[node distance=1.4cm,thick,scale=0.7,every node/.style={transform shape}]
\node[circle](AAAA0){};
\node[left=of AAAA0,circle](AAA0){};
\node[left=of AAA0,circle](AA0){};
\node[right=of AA0,circle,inner sep=9pt, label=45:](A0){};
\node[left=of AA0,circle,inner sep=5pt, label=45:](A1){};
\node[below=of A1,circle,draw,inner sep=9pt,label=-90:$G_1$](A2){$R$};
\node[left=of A2,circle,draw, inner sep=9pt, label=-90:$G_0$](A3){$R$};
\node[right=of A2,circle,draw, inner sep=9pt, label=-90:$G_2$](A4){$R$};
\node[right=of A4,circle,draw, inner sep=9pt, label=-90:$G_3$](A5){$R$};
\node[right=of A5,circle, inner sep=9pt](A6){};
\node[right=of A6,circle, inner sep=9pt](A7){};
\node[right=of A7,circle,draw, inner sep=9pt, label=-90:$G_{m-1}$](A8){$R$};
\draw (A2) to node[above]{$1$} (A3);
\draw (A2) to node[above]{$1$} (A4);
\draw(A4) to node[above]{$1$} (A5);
\draw(A5) to node [above]{$1$}(A6);
\draw(A7) to node [above]{$1$}(A8);
\draw[dashed] (A6) to (A7);
\draw (A3) to [bend left] node [above]{$x$}(A8);
\node[right=of A8,circle, draw, inner sep=9pt, label=-90:$G_0$](A9){$R$};
\node[right=of A9,circle, draw, inner sep=9pt, label=-90:$G_1$](A10){$R$};
\draw(A10) to node [above]{$x$}(A9);
\end{tikzpicture}
\end{center}
\caption{The graph $\Theta^m(G,R,x)$ for~\eqref{part1} and~\eqref{part3}  or $\Theta^2(G,R,x)$~for~\eqref{part2}} \label{FigFigFigFig}
\end{figure}
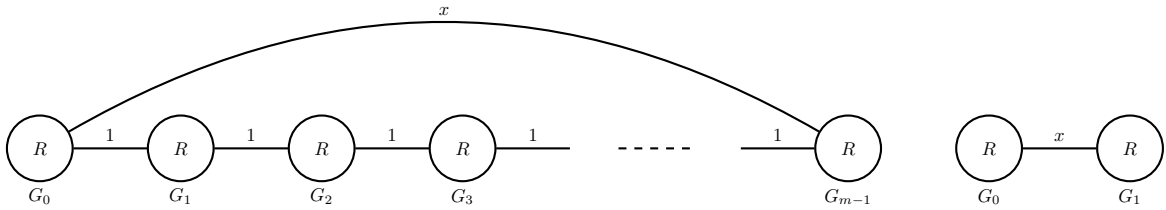

It is immediate, from the definition of $\Theta^m:=\Theta^m(G,R,x)$, that $[\Theta^m(g_i)]\cong [\Theta^m(h_j)]$, for every $g,h\in G$ and $i,j\in \{0,\ldots,m-1\}$. Indeed, for every $g\in G$ and for every $i\in \{0,\ldots,m-1\}$, we have
\[
[\Theta^m(g_i)]\cong
\begin{cases}
\G[g]\uplus \mathbf{K}_1\uplus \mathbf{K}_1&\mathrm{when }~\eqref{part1}\, \mathrm{ or }\,\eqref{part3} \,\mathrm{holds},\\
\G[g]\uplus \mathbf{K}_1&\mathrm{when }~\eqref{part2}\, \mathrm{ holds}.
\end{cases}
\]
%
%
}
\end{defi}

\begin{defi}{\rm
Let $G$ be a group, let $H$ be a subgroup of $G$ and let  $A$ be a subset of $G$.
The {\em coset digraph} $\mathrm{Cos}(G, H, A)$ is the digraph with vertex set
the set of right cosets of $H$ in $G$ and with arcs the ordered pairs $(Hx, Hy)$
such that $H{yx^{-1}} H \subseteq HAH$ (where $HAH := \{hsk \mid h, k \in H, s\in A\}$).
Since $\mathrm{Cos}(G, H, A) = \mathrm{Cos}(G, H, HAH)$,
replacing $A$ by $HAH$, we may (and we do) assume that $A$ is a union of
$H$-double cosets, that is, $A$ is a disjoint union $\bigcup_{s\in S} HsH$
for some subset $S$ of $G$. It is immediate to check that $\mathrm{Cos}(G, H, A)$
is undirected if and only if $A = A^{-1}$ and $\mathrm{Cos}(G, H, A)$ is
connected if and only if $G = \langle A\rangle$. Also, the action of $G$
by right multiplication on $G/H$ induces a vertex-transitive automorphism
group on $\mathrm{Cos}(G, H, A)$. Coset digraphs generalize the notion
of Cayley graph, which corresponds to the case $H=1$.}
\end{defi}
It was proved by Sabidussi~\cite{18} that every $G$-vertex-transitive graph
is isomorphic to some coset graph of $G$. More precisely, we have the following
well-known
result.

\begin{prop}\label{prop=sab}Let $\Gamma$ be a $G$-vertex-transitive graph
and let $\alpha$ be a vertex of $\Gamma$. Then there exists a union $A$ of
$G_\alpha$-double cosets such that $\Gamma \cong \mathrm{Cos}(G, G_\alpha , A)$
and with the action of $G$ on $V\Gamma$ equivalent to the action of $G$ by
right multiplication on the set of right cosets of $G_\alpha$ in $G$.
\end{prop}

The following lemma generalizes some classical results of Nowitz and Watkins~\cite{NowitzWatkins1} on Cayley graphs, and one may compare it with Lemma~\ref{Watkins-Nowitz}.

\begin{lem}\label{lem=partII0}
Let $\G:=\mathrm{Cos}(G,H,A)$ be a vertex-transitive graph and let $B$ be a subset of $A$. If, for every $\varphi\in \Aut(\G)_H$, $\varphi$ fixes $\{Hx\mid x\in B\}$ setwise, then $\varphi$ fixes $\{Hx\mid x\in \langle B\rangle\}$ setwise.
\end{lem}
\begin{proof}
Set $\mathcal{B}:=\{Hx\mid x\in \langle B\rangle\}\subseteq V\G$.
We define $B^0:=1$ and, for $i\in\mathbb{N}$ with $i\ge 1$,
we define $B^i:=\{b_1\cdots b_i\mid b_1,\ldots,b_i\in B\}$
the set consisting of all products of $i$ elements from $B$. Moreover,
for every $i\in\mathbb{N}$, we define $\mathcal{B}_{i}:=\{Hx:x\in B^i\}$.
Intuitively, when $i\ge 1$, $\mathcal{B}_{i}$ consists of the vertices of
$\G$ that can be reached from the vertices in $\mathcal{B}_{i-1}$ by
walking along the edges labeled from $B$. As $\langle B\rangle=\bigcup_i B^i$,
we have $\mathcal{B}=\bigcup_i\mathcal{B}_i$.

We now prove, by induction on $i\in\mathbb{N}$, that  for each $\varphi\in \Aut(\G)_H$, $\varphi$ fixes setwise $\mathcal{B}_i$. This is obvious when $i=0$, it is also true when $i=1$ because by hypothesis the elements of $\Aut(\G)_H$ fix setwise $\{Hx\mid x\in B\}=\mathcal{B}_1$.  Assume then $i\ge 2$. Let $\varphi\in \Aut(\G)_H$ and let $Hx\in \mathcal{B}_i$. From the definition of $\mathcal{B}_i$, we have $x=b_1b_2\cdots b_i$, for some $b_1,\ldots,b_i\in B$. Set $y:=b_2\cdots b_i\in B^{i-1}$; in particular, $x=b_1y$. By the inductive hypothesis, $(Hy)^\varphi\in\mathcal{B}_{i-1}$ and hence $(Hy)^\varphi=Hz$, for some $z\in B^{i-1}$. Now, consider the automorphism $\psi:=y\varphi z^{-1}\in \Aut(\G)$. We have
\begin{align*}
H^\psi&=(H)^{y\varphi z^{-1}}=(Hy)^{\varphi z^{-1}}=(Hz)^{z^{-1}}=H.
\end{align*}
In particular, $\psi\in \Aut(\G)_H$. Therefore, by hypothesis, $\psi$ fixes setwise $\mathcal{B}_1$. Since $Hb_1\in \mathcal{B}_1$, we deduce $(Hb_1)^\psi=Hb_1'$, for some $b_1'\in B^1$. On the other hand,
\begin{align*}
Hb_1'&=(Hb_1)^\psi=(Hb_1)^{y\varphi z^{-1}}=(Hb_1y)^{\varphi z^{-1}}=(Hx)^{\varphi z^{-1}}=(Hx)^{\varphi}z^{-1}.
\end{align*}
Thus $(Hx)^\varphi=Hb_1'z\in\mathcal{B}_i$, because $b_1'\in B^1$, $z\in B^{i-1}$ and $b_1'z\in B^i$. Since $Hx$ is an arbitrary element from $\mathcal{B}_i$, we deduce that $\varphi$ fixes setwise $\mathcal{B}_i$.
\end{proof}

\begin{lem}\label{lem=partII00}
Let $\G:=\mathrm{Cos}(G,H,A)$ be a vertex-transitive graph and let $B$ be a subset of $A$. If, for every $\varphi\in \Aut(\G)_H$, $\varphi$ fixes $\{Hx\mid x\in B\}$ pointwise, then $\varphi$ fixes $\{Hx\mid x\in \langle B\rangle\}$ pointwise.
\end{lem}
\begin{proof}
It follows verbatim the proof of Lemma~\ref{lem=partII0} replacing ``setwise" with ``pointwise''.
\end{proof}

\begin{lem}\label{lem=partII1}
Let $G$, $R$, $m$ and $x$ be as in Definition~$\ref{defi:22}$. For every $i\in \{0,\ldots,m-1\}$ and for every $\varphi\in \Aut(\Theta^m(G,R,x))$, we have $G_i^\varphi\in \{G_0,\ldots,G_{m-1}\}$.
\end{lem}
\begin{proof}
Set $\Theta:=\Theta^m(G,R,x)$, $A:=\Aut(\Theta)$, $\Sigma:=\Cay(G,R)$ and $i\in \{0,\ldots,m-1\}$.
Let $\mathcal{O}:=1_i^A=\{1_i^a\mid a\in A\}$ be the $A$-orbit containing $1_i$. As $G\le A$, we have $G_i\le \mathcal{O}$. Let $\G$ be the graph induced by $\Theta$ on $\mathcal{O}$ and let $H$ be the permutation group induced by $A$ on $\mathcal{O}$. Now, $\G$ is an $H$-vertex-transitive graph and hence (in view of Proposition~\ref{prop=sab}) we are in the position to apply Lemma~\ref{lem=partII0} to the graph $\G$, to the group $H\le \Aut(\G)$ and to the vertex $1_i$.

Now, $[\Theta(1_i)]\cong [\Sigma(1)]\uplus \mathbf{K}_1$ when~\eqref{part2} holds and $[\Theta(1_i)]\cong [\Sigma(1)]\uplus \mathbf{K}_1\uplus \mathbf{K}_1$ when~\eqref{part1} or~\eqref{part3} holds. Each  $\varphi\in \Aut(\G)_{1_i}$ fixes setwise the set $[\G(1_i)]_I$ of isolated vertices of $[\G(1_i)]$ and hence, it also fixes setwise its complement $[\G(1_i)]\setminus [\G(1_i)]_I=[\Theta(1_i)]\setminus [\Theta(1_i)]_I\cong\Sigma[R]\setminus \Sigma[R]_I$. From Definition~\ref{defi:22}, $G=\langle \Sigma[R]\setminus \Sigma[R]_I\rangle$ and hence, by Lemma~\ref{lem=partII0}, we infer that $\varphi$ fixes setwise $G_i=\langle [\Sigma[R]\setminus \Sigma[R]_I\rangle_i$. Since $\varphi$ is an arbitrary element of $\Aut(\G)_{1_i}$, we deduce that $H_{1_i}$ fixes setwise $G_i$ and hence $A_{1_i}$ fixes setwise $G_i$.

\smallskip

We now show that $A_{1_i}G=GA_{1_i}$. Let $\varphi\in A_{1_i}$ and $g\in G$. Then $1_i^{g\varphi }=g_i^\varphi \in G_i^{\varphi}=G_i$ and hence $1_i^{g\varphi}=h_i$, for some $h\in G$. Set $\psi:=g\varphi h^{-1}$. Then $1_i^{\psi}=1_i^{g\varphi h^{-1}}=h_i^{h^{-1}}=1_i$, that is, $\psi\in A_{1_i}$. Moreover, $g\varphi =\psi h$. This shows that $GA_{1_i}\subseteq A_{1_i}G$. The other inclusion follows from the fact that $|GA_{1_i}|=|G||A_{1_i}|=|A_{1_i}G|$.

As $A_{1_i}G=GA_{1_i}$, we obtain that $A_{1_i}G$ is subgroup of $A$ with $A_{1_i}\le A_{1_i}G\le A$. Therefore, $1_i^{A_{1_i}G}=1_i^{G}=G_i$ is a block of imprimitivity for the action of $A$, see~\cite[Theorem~$1.5A$]{DixonMortimer}.

Since this argument holds for each $i$, we deduce that $G_i^\varphi\in \{G_0,\ldots,G_{m-1}\}$ for every $i\in \{0,\ldots,m-1\}$.
\end{proof}

\begin{lem}\label{lem=partII2}
Let $G$, $R$, $m$ and $x$ be as in Definition~$\ref{defi:22}$. The graph $\Theta^m(G,R,x)$ is an $m$-$\mathrm{GRR}$ for $G$.
\end{lem}
\begin{proof}
Set $\Theta:=\Theta^m(G,R,x)$ and $A:=\Aut(\Theta)$.
From Lemma~\ref{lem=partII1}, $A$ fixes setwise $\{G_0,\ldots,G_{m-1}\}$
and hence $A$ induces an action on the indexed set $\{0,\ldots,m-1\}$.
From the structure of $\Theta$, we deduce that the action of $A$ on the
indexed set $\{0,\ldots,m-1\}$ induces a subgroup of the dihedral group
$\langle (0,\ldots,m-1),(1,m-1)(2,m-2)\cdots\rangle$. (For the rest of this proof, we consider the indexed set $\{0,\ldots,m-1\}$
as the integers modulo $m$, that is, $G_{i}=G_{i+m}$.)

\smallskip
\noindent\textsc{Claim 1: }For every $\varphi\in A$, there exists $
i\in \{0,\ldots,m-1\}$ and $h\in G$ such that $g_0^{\varphi h}=g_i$, for every $g\in G$.
\smallskip

\noindent Indeed, let $i\in \{0,\ldots,m-1\}$ with $G_0^\varphi=G_i$. Then $g_0^\varphi=(g^{\varphi'})_i$, for some permutation $\varphi':G\to G$. Observe that, if $a$ and $b$ are adjacent in $\Cay(G,R)$, then $a_0$ and $b_0$ are adjacent in $\Theta$ and hence $a_0^\varphi=(a^{\varphi'})_i$ and $b_0^\varphi=(b^{\varphi'})_i$ are adjacent in $\Theta$, that is, $a^{\varphi'}$ and $b^{\varphi'}$ are adjacent in $\Cay(G,R)$. This shows that $\varphi'\in \Aut(\Cay(G,R))$. Since $\Cay(G,R)$ is a GRR, there exists $h\in G$ with $g^{\varphi'}=gh^{-1}$, for every $g\in G$. Therefore, $g_0^{\varphi h}=((g^{\varphi'})_i)^h=(gh^{-1})_i^h=g_i$, for every $g\in G$.~$_\blacksquare$

\smallskip

\noindent\textsc{Claim 2: }Let $\varphi\in A$. There exist no $i,j\in \{0,\ldots,m-1\}$ and $y\in G\setminus \mathbf{Z}(G)$ such that $g_i^{\varphi}=(yg)_j$, for every $g\in G$.

\smallskip

\noindent We suppose that there exist
$i,j\in \{0,\ldots,m-1\}$ and $y\in G$ with $g_i^\varphi=(yg)_j$,
for every $g\in G$. Let $a,b\in G$ with $ba^{-1}\in R$. Then $a_i$
and $b_i$ are adjacent in $\Theta$ and hence $a_i^\varphi=(ya)_j$
 and $b_i^\varphi=(yb)_j$ are adjacent in $\Theta$. Therefore, $(yb)(ya)^{-1}\in R$, that is, $yba^{-1}y^{-1}\in R$. Since $a$ and $b$ are arbitrary adjacent
 vertices of $\Cay(G,R)$, this shows $R^{y^{-1}}=R$. In particular,
 the inner automorphism of $G$ via $y^{-1}$ is an automorphism of
 $\Cay(G,R)$ fixing the identity. Since $\Cay(G,R)$ is a GRR,
 we deduce that $ygy^{-1}=g$, for each $g\in G$. Therefore
 $y\in \mathbf{Z}(G)$.~$_\blacksquare$

\smallskip
Now, we argue by contradiction and we suppose that $A$ contains
some permutation $\varphi$ not fixing each block $G_0,\ldots,G_{m-1}$:
we distinguish two cases depending on whether $\varphi$ acts as
a rotation or as a reflection on  $\{G_0,\ldots,G_{m-1}\}$.
From Claim~1, replacing $\varphi$ by $\varphi h$ for some suitable $h\in G$, we may assume that $g_0^\varphi=g_i$, for every $g\in G$.
\smallskip

The case $m=2$ is slightly degenerate and hence we study in detail this first. In particular, $g_0^\varphi=g_1$, for every $g\in G$. Assume first that~\eqref{part2} holds. Now, for each $g\in G$, $(xg)_1$ is the only vertex in $G_1$ adjacent to $g_0\in G_0$, therefore $(xg)_1^\varphi$ is the only vertex of $G_0$ adjacent to $g_0^\varphi=g_1$. However, $(x^{-1}g)_0$ is the only vertex in $G_0$ adjacent to $g_1$. Therefore, $(xg)_1^{\varphi}=(x^{-1}g)_0$. This shows that $g_1^\varphi=(x^{-2}g)_0$, for every $g\in G$. Now, Claim~2 yields $x^2\in \mathbf{Z}(G)$, contradicting our choice of $x$ in Definition~\ref{defi:22}~\eqref{part2}.

Suppose next that~\eqref{part3} holds. For each $g\in G$, $g_1^\varphi\in G_0$ and hence there exists a bijection $\varphi':G\to G$ such that $g_1^\varphi=(g^{\varphi'})_0$, for every $g\in G$. Arguing as in Claim~1, we see that $\varphi'\in\Aut(\Cay(G,R))$. Since $\Cay(G,R)$ is a GRR, there exists $y\in G$ such that $g^{\varphi'}=gy$, for every $g\in G$. Therefore, $g_1^\varphi=(gy)_0$, for every $g\in G$. Summing up,
\begin{align*}
g_0^\varphi&=g_1\, \, \mbox{and} \,\,g_1^\varphi=(gy)_0,\ \mbox{for every $g\in G$}.
\end{align*}

Since $1_0$ is adjacent to $1_1$ and to $x_1$, we obtain that $1_1=1_0^\varphi$ is adjacent to $1_1^\varphi=y_0$ and to $x_1^\varphi=(xy)_0$. However, the neighbors of $1_1$ in $G_0$ are $1_0$ and $(x^{-1})_0$. This yields either $y_0=1_0$ and $(xy)_0=(x^{-1})_0$, or $y_0=(x^{-1})_0$ and $(xy)_0=1_0$. In the former case, we have $y=1$ and $x=x^{-1}$, that is, $o(x)=2$, contradicting our choice of $x$ in Definition~\ref{defi:22}~\eqref{part3}. In the latter case, $y=x^{-1}$. Now, let $g\in G$. Since $g_0$ is adjacent to $g_1$ and to $(xg)_1$, we obtain that $g_1=g_0^\varphi$ is adjacent to $g_1^\varphi=(gx^{-1})_0$ and to $(xg)_1^\varphi=(xgx^{-1})_0$. Since $x\ne 1$, we cannot have $g=gx^{-1}$. Therefore, $g=xgx^{-1}$. Since this argument does not depend upon $g\in G$, we deduce $x\in\mathbf{Z}(G)$, contradicting our choice of $x$ in Definition~\ref{defi:22}~\eqref{part3}.

\smallskip

Now, we assume  $m\ge 3$. We start by assuming that $\varphi$ acts as a rotation. Therefore, there exists a divisor $\ell$ of $m$  with $G_i^\varphi=G_{i+\ell}$, for every $i\in \{0,\ldots,m-1\}$.  Thus
$g_0^\varphi=g_\ell$, for each $g\in G$.

For every $g\in G$, $g_1$ is the only vertex in $G_1$ adjacent to $g_0\in G_0$, therefore $g_1^\varphi$ is the only vertex in $G_1^\varphi=G_{\ell+1}$ adjacent to $g_0^\varphi=g_\ell\in G_\ell$.  From the structure of $\Theta$, we deduce that $g_1^\varphi=g_{\ell+1}$. Arguing inductively in a similar fashion and using the cycle-like structure of $\Theta$, for every $g\in G$ and for every $i\in \{0,\ldots,m-1\}$, we obtain
\[
g_i^\varphi=
\begin{cases}
g_{i+\ell}&\textrm{if }i\in \{0,\ldots,m-\ell-1\},\\
(x^{-1}g)_{i+\ell}&\textrm{if }i\in \{m-\ell,\ldots,m-1\}.
\end{cases}
\]
In particular, $g_{m-\ell}^\varphi=(x^{-1}g)_0$, for each $g\in G$. Since $x\in G\setminus\mathbf{Z}(G)$, from Claim~2 we obtain a contradiction because $\varphi$ is  an  automorphism of $\Theta$.

\smallskip

Suppose now that $\varphi$ acts as a reflection on $\{G_0,\ldots,G_{m-1}\}$. In this case, there exists $\ell\in\{0,\ldots,m-1\}$ such that $G_i^\varphi=G_{m-i-\ell}$, for every $i\in \{0,\ldots,m-1\}$. Now, $g_0^\varphi=g_{m-\ell}$, for each $g\in G$. For every $g\in G$, $g_1$ is the only vertex in $G_1$ adjacent to $g_0\in G_0$, therefore $g_1^\varphi$ is the only vertex in $G_1^\varphi=G_{m-\ell-1}$ adjacent to $g_0^\varphi=g_{m-\ell-1}\in G_\ell$. From the structure of $\Theta$, we deduce that $g_1^\varphi=g_{m-\ell-1}$. Arguing inductively and using the cycle-like structure of $\Theta$, for every $g\in G$ and for every $i\in \{0,\ldots,m-1\}$, we obtain
\[
g_i^\varphi=
\begin{cases}
g_{m-\ell-i}&\textrm{if }i\in \{0,\ldots,m-\ell\},\\
(xg)_{m-\ell-i}&\textrm{if }i\in \{m-\ell+1,\ldots,m-1\}.
\end{cases}
\]
In the very special case that $\ell=1$, observe that $m-\ell+1=m$ and hence the second row in the previous formula does not arise. Assume first $\ell\ne 1$. In particular, $g_{m-\ell+1}^\varphi=(xg)_{m-1}$, for each $g\in G$. Since $x\in G\setminus\mathbf{Z}(G)$, from Claim~2 we obtain a contradiction because $\varphi$ is an  automorphism of $\Theta$.
Finally, assume $\ell=1$. Therefore $g_i^{\varphi}=g_{m-1-i}$, for every $g\in G$ and for every $i\in \{0,\ldots,m-1\}$. For every $g\in G$, $\{g_0,(xg)_{m-1}\}$ is an edge of $\Theta$ and hence so is $\{g_0^\varphi,(xg)_{m-1}^\varphi\}=\{g_{m-1},(xg)_0\}$. However, the only edges between $G_0$ and $G_{m-1}$ are of the form $\{h_0,(xh)_{m-1}\}$. Therefore, we have $x^2g=g$, that is, $x^2=1$, contradicting our choice of $x$ in Definition~\ref{defi:22}.

\smallskip

Summing up, we have proved that $G_i^\varphi=G_i$, for every $i\in \{0,\ldots,m-1\}$ and for every $\varphi\in \Aut(\Theta)=A$. Assume now that~\eqref{part1} or~\eqref{part2} in Definition~\ref{defi:22} holds.  Let $\varphi\in A_{1_0}$. Since $G_0^\varphi=G_0$, $\varphi$ induces an automorphism of the GRR $\Cay(G,R)$. Since $\varphi$ fixes the identity element, we obtain that $\varphi$ fixes pointwise $G_0$. Since $\varphi$ fixes setwise $G_1$ and since between $G_0$ and $G_1$ there is a complete matching, we infer that $\varphi$ fixes pointwise $G_1$. Now, an easy inductive argument shows that $\varphi=1$ and hence $\Theta$ is an $m$-GRR for $G$. Finally, suppose that ~\eqref{part3} in Definition~\ref{defi:22} holds. Let $\varphi\in A_{1_0}$. Then $\varphi$ induces an automorphism of the GRR $\Theta[G_0]\cong\Cay(G,R)$ and hence $\varphi$ fixes pointwise $G_0$, because $\varphi$ fixes $1_0$. Similarly, $\varphi$ induces an automorphism of the GRR $\Theta[G_1]\cong \Cay(G,R)$ and hence there exists $y\in G$ with $g_1^\varphi=(gy)_1$, for every $g\in G$.
Since $1_0$ is adjacent only to $1_1$ and to $x_1$ in $G_1$, we obtain that $1_0=1_0^\varphi$ is adjacent to $1_1^\varphi=y_1$ and to $x_1^\varphi=(xy)_1$, and hence $\{1_1,x_1\}=\{y_1,(xy)_1\}$. This yields either $y=1$ and $x=xy$, or $x=y$ and $1=xy$.
In the first case, $\varphi=1$, and in the second case, $x^2=1$, contradicting our choice of $x$ in Definition~\ref{defi:22}~\eqref{part3}.
\end{proof}

In the light of Lemma~\ref{lem=partII2}, to conclude our analysis on the existence of $m$-GRRs for  non-abelian groups $G$ admitting a GRR, we need to deal with $m=2$ and only for the groups $G$ not satisfying~\eqref{part1},~\eqref{part2} or~\eqref{part3} in Definition~\ref{defi:22}. We do this in the rest of this section.

\begin{lem}\label{lem=new}Let $G$ be a non-abelian group admitting a $\mathrm{GRR}$. Then $G$ admits a $2$-$\mathrm{GRR}$.
\end{lem}
\begin{proof}
Let $R$ be a Cayley subset of $G$ with $\Sigma:=\Cay(G,R)$ a GRR and with $G=\langle R\setminus \Sigma[R]_I\rangle$: the existence of $R$ is proved in Lemma~\ref{lem=R_I}.
By Lemma~\ref{lem=partII2}, we may assume that Definition~\ref{defi:22}~\eqref{part2} and~\eqref{part3} are not satisfied by $G$. In particular, since~\eqref{part2} does not hold, $g^2\in \mathbf{Z}(G)$, for every $g\in G$. Therefore, $G/\mathbf{Z}(G)$ is an elementary abelian $2$-group. Thus $G=O\times P$, for some abelian group $O$ of odd order and for some non-abelian $2$-group $P$ with $p^2\in \mathbf{Z}(P)$ for each $p\in P$. Moreover, since Definition~\ref{defi:22}~\eqref{part3} does not hold, $G=R\cup\mathbf{Z}(G)\cup\{g\in G\mid o(g)=2\}$.

\smallskip

\noindent\textsc{Claim: }There exists a Cayley subset $L$ of $G$ such that $|L|=|R|$ and with $R$ and $L$ {\em not} having the same number of involutions.

\smallskip

\noindent Let $\ell$ be the number of involutions in $R$ and let $2\kappa$ be the number of non-involutions in $R$. Similarly, let $\ell'$ be the number of involutions in $G$ and let $2\kappa'$ be the number of non-involutions in $G\setminus\{1\}$. As $G$ has even order, $\ell'$ is odd. If $\ell'=1$, then $P$ is either cyclic or dicyclic, see~\cite[5.3.6]{Robinson}. In the first case, $G=O\times P$ is abelian, which is a contradiction. Suppose $P$ is dicyclic. As $p^2\in \mathbf{Z}(P)$ for each $p\in P$, we deduce $P\cong Q_8$ and hence $G=O\times Q_8$. Therefore, $\{g\in G\mid o(g)=2\}\subseteq \mathbf{Z}(G)$; hence $G=R\cup \mathbf{Z}(G)\cup \{g\in G\mid o(g)=2\}=R\cup\mathbf{Z}(G)$ and $R^c\subseteq \mathbf{Z}(G)$.  Since $\Cay(G,R)$ is a GRR, so is $\Cay(G,R^c)$; however this is a contradiction because $\langle R^c\rangle\le \mathbf{Z}(G)<G$ and $\Cay(G,R^c)$  is disconnected. This shows $\ell'\ge 3$

When $\ell\le \ell'-2$ and $\kappa\ge 1$, we may remove $x$ and $x^{-1}$ from $R$, for some $x\in R$ with $o(x)>2$, and we may add two involutions from $G\setminus R$ to form our Cayley subset $L$. When $\ell\le \ell'-2$ and $\kappa=0$, we may remove two involutions from $R$ and we may add $x$ and $x^{-1}$, for some $x\in G\setminus R$ with $o(x)>2$,  to form our Cayley subset $L$. In particular, we may suppose that $\ell>\ell'-2$, that is, $\ell=\ell'$ or $\ell=\ell'-1$.

When $\ell\ge\ell'-1$ and $\kappa<\kappa'$, remove two involutions from $R$ (which is possible because $\ell'-1\ge 2$) and add $x$ and $x^{-1}$ with $x\in G\setminus R$ and $o(x)>2$ to form our Cayley subset $L$. Finally, when $\ell\ge \ell'-1$ and $\kappa'=\kappa$, we have $|R|\ge |G|-2$. Thus $R$ contains all non-identity elements of $G$ but one. However, it is easy to see that this contradicts the fact that $\Cay(G,R)$ is a GRR.~$_\blacksquare$

\smallskip

Let $L$ be a Cayley subset of $G$ with $|L|=|R|$ and with $R$ and $L$ not having the same number of involutions. Consider $\Gamma:=\BiCay(G,R,L,S)$ with $S=\{1\}$, and let $A:=\Aut(\Gamma)$.

Now, $[\Gamma(1_0)]\cong [\Sigma(1)]\uplus \mathbf{K}_1$. Each  $\varphi\in A_{1_0}$ fixes setwise the set $[\G(1_0)]_I$ of isolated vertices of $[\G(1_0)]$ and hence, it also fixes setwise its complement $[\G(1_0)]\setminus [\G(1_0)]_I\cong[\Sigma(1)]\setminus [\Sigma(1)]_I\cong\Sigma[R]\setminus \Sigma[R]_I$. As $G=\langle \Sigma[R]\setminus \Sigma[R]_I\rangle$, by Lemma~\ref{lem=partII0}, we infer that $\varphi$ fixes setwise $G_0=\langle \G(1_0)\setminus [\G(1_0)]_I\rangle_0$. Since $\varphi$ is an arbitrary element of $A_{1_0}$, we deduce that $A_{1_0}$ fixes setwise $G_0$.

The action of $A_{1_0}$ on $G_0$ induces a group of automorphisms on the GRR $\Cay(G,R)$. Since $A_{1_0}$ fixes the vertex $1_0$, we deduce that $A_{1_0}$ fixes pointwise $G_0$. As $|S|=1$, there is a perfect matching between $G_0$ and $G_1$ and hence $A_{1_0}$ fixes pointwise $G_1$. Thus $A_{1_0}=1$.

From the previous paragraph, we infer that either $A=G$ and $\G$ is a $2$-GRR for $G$, or  $A$ acts regularly on the vertices of $\G$. In the first case we are done and hence we assume that the latter case holds. As $G\unlhd A$, $\{G_0,G_1\}$ is a system of imprimitivity for $A$. Let $\varphi\in A$ with $1_1^\varphi=1_0$. For each $g\in G$, there exists $g'\in G$ with $g_1^\varphi=g_0'$. As $\varphi$ is a graph automorphism, the mapping $\varphi':\Cay(G,L)\to \Cay(G,R)$ defined by $g^{\varphi'}=g'$ (for each $g\in G$) is a graph isomorphism. In particular, since $\Cay(G,R)$ is a GRR, so is $\Cay(G,L)$. Moreover, since $1_1^\varphi=1_0$, we have $1^{\varphi'}=1$.

The isomorphism $\varphi'$ between $\Cay(G,L)$ and $\Cay(G,R)$ induces an isomorphism between the corresponding automorphism groups. Namely, the mapping $a\mapsto \varphi'^{-1}a\varphi'$ defines an isomorphism between $\Aut(\Cay(G,L))=G$ and $\Aut(\Cay(G,R))=G$. Therefore $G^{\varphi'}=G$. This shows that $\varphi'$ normalizes $G$ in the symmetric group $\mathrm{Sym}(G)$. As $1^{\varphi'}=1$, we deduce $\varphi'\in\Aut(G)$.

The neighborhood of $1$ in $\Cay(G,L)$ is $L$ and the neighborhood of $1^{\varphi'}=1$ in $\Cay(G,R)$ is $R=L^{\varphi'}$. Therefore $L$ and $R$ are Cayley subsets conjugate via an element of $\Aut(G)$. In particular, $L$ and $R$ have the same number of involutions, contradicting our choice of $L$.
\end{proof}

\begin{cor}\label{cor=part1}Let $G$ be a non-abelian group admitting a $\mathrm{GRR}$ and let $m\ge 1$ be an integer. Then $G$ admits an $m$-$\mathrm{GRR}$.
\end{cor}
\begin{proof}It follows immediately from Definition~\ref{defi:22} and Lemmas~\ref{lem=partII2} and~\ref{lem=new}.
\end{proof}

\section{Groups admitting no GRR}\label{sec4}

In this section, we prove Theorem~\ref{theo=main} for groups admitting no GRR and for abelian groups admitting GRRs (that is, elementary abelian $2$-groups): we use Proposition~\ref{prop=GRR}, that is, Godsil's classification of groups admitting no GRR.

\begin{defi}\label{defi:11}
{\rm Let $G$ be a group, let $m$ be a positive integer with $m\ge 3$, let $R$, $L$ and  $T$ be Cayley subsets of $G$, let $S$ be a subset of $G$ with $|R|=|L|=|T|-|S|+1$, and let $x\in G\setminus S$.  We let $\Theta^m(G,R,L,S,T,x)$ be the $m$-Cayley graph with vertex set $G\times \{0,\ldots,m-1\}$ and with edge set consisting of the pairs:
\begin{description}
\item[(i)] $\{g_0,(rg)_0\}$, for every $g\in G$ and $r\in R$;
\item[(ii)] $\{g_1,(lg)_1\}$, for every $g\in G$ and $l\in L$;
\item[(iii)] $\{g_0,(sg)_1\}$, for every $g\in G$ and $s\in S$;
\item[(iv)] $\{g_i,(tg)_i\}$, for every $i\in \{2,\ldots,m-1\}$, $g\in G$ and $t\in T$;
\item[(v)] $\{g_i,g_{i+1}\}$, for every $g\in G$ and $i\in \{1,\ldots,m-2\}$;
\item[(vi)] $\{g_0,(xg)_{m-1}\}$, for every $g\in G$.
\end{description}
Observe that $\Theta^m(G,R,L,S,T,x)$ is regular of valency $|R|+|S|+1=|L|+|S|+1=|T|+2$.}
\end{defi}

Figure~\ref{FigFigFig} might be of some help for familiarizing with the structure of $\Theta^m(G,R,L,S,T,x)$.
\begin{figure}[!hhh]
\begin{center}
\begin{tikzpicture}[node distance=1.4cm,thick,scale=0.7,every node/.style={transform shape}]
\node[circle](AAAA0){};
\node[left=of AAAA0,circle](AAA0){};
\node[left=of AAA0,circle](AA0){};
\node[right=of AA0,circle,inner sep=9pt, label=45:](A0){};
\node[left=of AA0,circle,inner sep=5pt, label=45:](A1){};
\node[below=of A1,circle,draw,inner sep=9pt,label=-90:$G_1$](A2){$L$};
\node[left=of A2,circle,draw, inner sep=9pt, label=-90:$G_0$](A3){$R$};
\node[right=of A2,circle,draw, inner sep=9pt, label=-90:$G_2$](A4){$T$};
\node[right=of A4,circle,draw, inner sep=9pt, label=-90:$G_3$](A5){$T$};
\node[right=of A5,circle, inner sep=9pt](A6){};
\node[right=of A6,circle, inner sep=9pt](A7){};
\node[right=of A7,circle,draw, inner sep=9pt, label=-90:$G_{m-1}$](A8){$T$};
\draw (A2) to node[above]{$S$} (A3);
\draw (A2) to node[above]{$1$} (A4);
\draw(A4) to node[above]{$1$} (A5);
\draw(A5) to node [above]{$1$}(A6);
\draw(A7) to node [above]{$1$}(A8);
\draw[dashed] (A6) to (A7);
\draw (A3) to [bend left] node [above]{$x$}(A8);
\end{tikzpicture}
\end{center}
\caption{The graph $\Theta^m(G,R,L,S,T,x)$~}\label{FigFigFig}
\end{figure}
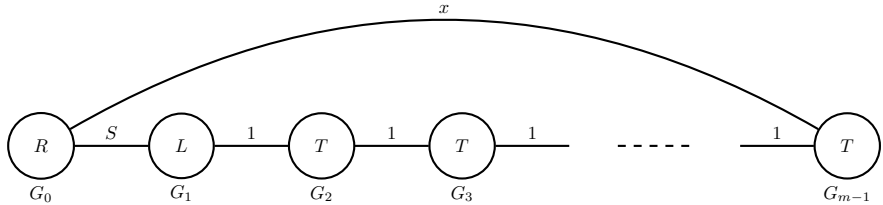

It is immediate from the definition of $\Theta^m:=\Theta^m(G,R,L,S,T,x)$ that $[\Theta^m(g_i)]\cong [\Theta^m(h_j)]$, for every $g,h\in G$ and $i,j\in \{2,\ldots,m-1\}$.

\begin{lem}\label{lem=prel}Let $G,R,L,T,S$ and $x$ be as in Definition~$\ref{defi:11}$  and let $\Theta:=\Theta^m(G,R,L,S,T,x)$. Suppose that
\begin{enumerate}
\item\label{enu2} $\BiCay(G,R,L,S)$ is a $2$-$\mathrm{GRR}$;
\item\label{enu3} for each $i\in \{2,\ldots,m-1\}$, we have $[\Theta(1_0)]\ncong [\Theta(1_i)]$ and $[\Theta(1_1)]\ncong [\Theta(1_i)]$.
\end{enumerate}
Then $\Theta$ is an $m$-$\mathrm{GRR}$.
\end{lem}
\begin{proof}
Let $A:=\Aut(\Theta)$. For every $i\in \{0,\ldots,m-1\}$, $g\in G$ and $a\in A$, we have $[\Theta(g_i)]\cong [\Theta(1_i)]$ and $[\Theta(g_i^a)]\cong [\Theta(g_i)]$. Therefore, from~\eqref{enu3}, $A$ fixes setwise $G_0\cup G_1$ and hence $A$ acts as a group of automorphisms on the Bi-Cayley graph $\BiCay(G,R,L,S)$.  By~\eqref{enu2}, $A_{1_0}$ fixes pointwise $G_0\cup G_1$. Since there is a perfect matching between $G_0$ and $G_{m-1}$, and between $G_1$ and $G_2$, we obtain that $A_{1_0}$ fixes pointwise $G_{m-1}$ and $G_2$. Since there is a perfect matching between $G_{i-1}$ and $G_{i}$, for every $i\in \{2,\ldots,m-1\}$, we obtain (arguing inductively) that $A_{1_0}$ fixes pointwise $G_i$. Therefore, $A_{1_0}=1$ and hence $A=G$, that is, $\Theta$ is an $m$-GRR for $G$.
\end{proof}

\begin{lem}\label{lem=prel2}Let $G,R,L,T,S$ and $x$ be as in Definition~$\ref{defi:11}$,  let $\Theta^m:=\Theta^m(G,R,L,S,T,x)$ and let $\Theta^3:=\Theta^3(G,R,L,S,T,x)$. Then $[\Theta^m(1_0)]\cong [\Theta^3(1_0)]$, $[\Theta^m(1_1)]\cong [\Theta^3(1_1)]$ and, for each $i\in\{2,\ldots,m-1\}$, $[\Theta^m(1_i)]\cong[\Theta^3(1_2)]$.
\end{lem}
\begin{proof}
The proof explains the central role of the element $x$ chosen in Definition~\ref{defi:11}. If $m=3$, then there is nothing to prove, therefore we may assume $m\ge4$. Recall that $[\Theta^m(1_i)]\cong[\Theta^m(1_{m-1})]$ for each $i\in \{2,\ldots,m-1\}$. We have
\begin{align*}
\Theta^m(1_0)&=\{r_0\mid r\in R\}\cup\{s_1\mid s\in S\}\cup \{x_{m-1}\},\\
\Theta^m(1_1)&=\{l_1\mid l\in L\}\cup\{(s^{-1})_0\mid s\in S\}\cup \{1_{2}\},\\
\Theta^m(1_{m-1})&=\{t_{m-1}\mid t\in T\}\cup\{1_{m-2}\}\cup \{(x^{-1})_{0}\},\\
\Theta^3(1_0)&=\{r_0\mid r\in R\}\cup\{s_1\mid s\in S\}\cup \{x_{2}\},\\
\Theta^3(1_1)&=\{l_1\mid l\in L\}\cup\{(s^{-1})_0\mid s\in S\}\cup \{1_{2}\},\\
\Theta^3(1_2)&=\{t_{2}\mid t\in T\}\cup\{1_{1}\}\cup \{(x^{-1})_{0}\}.
\end{align*}
In the graph $\Theta^m$, $x_{m-1}$ is not adjacent to the elements in $G_1$ because $m\ge 4$ and, in the graph $\Theta^3$, $x_2$ is not adjacent to the elements in $\{s_1\mid s\in S\}$  because  $x\notin S$. Therefore $[\Theta^m(1_0)]\cong[\Theta^3(1_0)]$. It is clear that $[\Theta^m(1_1)]\cong[\Theta^3(1_1)]$. As above, in the graph $\Theta^m$, $(x^{-1})_0$ is not adjacent to the elements in $G_{m-2}$ because $m\ge 4$ and, in the graph $\Theta^3$, $(x^{-1})_0$ is not adjacent to $1_1$ because $x\notin S$. Therefore $[\Theta^m(1_{m-1})]\cong[\Theta^3(1_2)]$.
\end{proof}

Lemma~\ref{lem=prel2} is important theoretically and computationally, indeed, it allows to check hypothesis~\eqref{enu3} in Lemma~\ref{lem=prel} only when $m=3$ and then deduce it for every integer greater than $3$.

Nowitz and Watkins, in their work on the GRR problem,  proved a lemma that is very useful in our context also (compare with Lemmas~\ref{lem=partII0} and~\ref{lem=partII00}).

\begin{lem}[Nowitz and Watkins~\cite{NowitzWatkins1}]\label{Watkins-Nowitz}
Let $G$ be a group, let $S$ be a subset of $G$, let $\Gamma:=\Cay(G,S)$ and let $X$ be a subset of $S$. If $\varphi$ fixes $X$ pointwise for every $\varphi\in \Aut(\Gamma)_1$, then $\varphi$ fixes $\langle X\rangle$ pointwise for every $\varphi\in \Aut(\Gamma)_1$. In particular, $\Aut(\Gamma)_1=1$ if $G=\langle X\rangle$ or if $\Gamma[S]$ is asymmetric.
\end{lem}

\subsection{Part 1: Cyclic groups and dicyclic groups}
\begin{notation}\label{hyp1}{\rm We set some notation that we use in this section: $G:=\langle a\rangle$ is either a cyclic group or $G:=\langle a,b\rangle $ is a dicyclic group over the cyclic group $\langle a\rangle$ with $a^b=a^{-1}$.}
\end{notation}
\begin{lem}\label{lem=cyclic}
Let $m$ be a positive integer with $m\ge 2$ and let $G$ be as in Notation~$\ref{hyp1}$. If $G$ is cyclic and $|G|\ge 6$, then $G$ admits an $m$-$\mathrm{GRR}$; if $G$ is dicyclic and $G\ncong Q_8$, then $G$ admits an $m$-$\mathrm{GRR}$.
\end{lem}

\begin{proof}
Suppose first that $G:=\langle a\rangle$ is a cyclic group of order $n\ge 6$. Let $R:=\{a,a^{-1}\}$, $L:=\{a^2,a^{-2}\}$, $S:=\{1,a,a^3\}$, $T:=\{a,a^{-1},a^2,a^{-2}\}$ and $x:=a^2$. Observe that $R=R^{-1}$, $L=L^{-1}$, $T=T^{-1}$, $|R|=|L|=|T|-|S|+1$ and $x\notin S$.
Let $\Theta:=\Theta^m(G,R,L,S,T,x)$, $\G:=\BiCay(G,R,L,S)$ and let $A:=\Aut(\G)$.

We start by proving that $\G$ is a $2$-GRR. When $n\in \{6,7,8,9\}$, a computation with \texttt{magma} shows that $\G$ is a $2$-GRR. Therefore, for the rest of our argument, we may assume $n\ge 10$.

It is easy to check that
\begin{align*}
\G_1(a^i_0)&=\{a^i_1,a^{i-1}_0,a^{i+1}_0,a^{i+1}_1,a^{i+3}_1\},\quad\textrm{ for every }i\in\{0,\cdots,n-1\},\\
\G_1(a^i_1)&=\{a^{i-2}_1, a^{i-3}_0,a^i_0,a^{i-1}_0,a^{i+2}_1\},\quad\textrm{ for every }i\in\{0,\cdots,n-1\}, \\
  \G_2(1_0)&=\{a_0^{-3},a_0^{-2},a_0^2,a_0^3,a_1^{-2},a_1^{-1},a_1^2,a_1^4,a_1^5\},  \\
  \G_2(1_1)&=\{a_0^{-5},a_0^{-4},a_0^{-2},a_0,a_0^2,a_1^{-4},a_1^{-3},a_1^{-1},a_1,a_1^3,a_1^4\}.
\end{align*}
Moreover, the induced subgraphs $[\G(1_0)]$ and $[\G(1_1)]$ are as in Figure~\ref{Fig}.
\begin{figure}[htbp]
\begin{center}
\unitlength 4mm
\begin{picture}(15,17)
\put(-2.8, 15.2){\line(1, 0){4}}
\put(-2.8, 13.2){\line(1, 0){4}}
\put(1.2, 13.2){\line(0, -1){2}}

\put(-3, 15){$\bullet$}
\put(-3, 13){$\bullet$}

\put(1, 15){$\bullet$}
\put(1, 13){$\bullet$}
\put(1, 11){$\bullet$}

\put(-4.8, 15){$a^{-1}_0$}
\put(-4.3, 13){$a_0$}

\put(2, 15){$1_1$}
\put(2, 13){$a_1$}
\put(2, 11){$a^3_1$}

\put(-3, 9){$[\G(1_0)]$}

\put(15.2, 15.2){\line(0, -1){2}}
\put(15.2, 13.2){\line(1,  0){4}}
\put(15.2, 11.2){\line(1,  0){4}}

\put(15, 15){$\bullet$}
\put(15, 13){$\bullet$}
\put(15, 11){$\bullet$}

\put(19, 13){$\bullet$}
\put(19, 11){$\bullet$}

\put(13.5, 15){$1_0$}
\put(13.2, 13){$a^{-1}_0$}
\put(13.2, 11){$a^{-3}_0$}

\put(20, 13){$a^2_1$}
\put(20, 11){$a^{-2}_1$}

\put(16, 9){$[\G(1_1)]$}
\end{picture}
\end{center}\vspace{-3.5cm}
\caption{The induced subgraphs $[\G(1_0)]$ and $[\G(1_1)]$}\label{Fig}
\end{figure}

Since $|\G_2(1_0)|=9$ and $|\G_2(1_1)|=11$, we obtain that $A$ has
two orbits on $V\G$, that is, $G_0$ and $G_1$.
Since $A$ fixes setwise $G_0$ and $G_1$, $A_{1_0}$ fixes setwise $\Gamma(1_0)\cap G_0$ and $\G(1_0)\cap G_1$ and $A_{1_1}$ fixes  setwise $\Gamma(1_1)\cap G_0$ and $\G(1_1)\cap G_1$. Now, by looking at  Figure~\ref{Fig}, we deduce that  $A_{1_0}$ fixes pointwise $\G(1_0)$ and $A_{1_1}$ fixes pointwise $\G(1_1)$. As $\G$ is connected, an easy connectedness argument implies that $A_{1_0}=A_{1_1}=1$, that is, $A=G$ and $\G$ is a $2$-GRR.

In the light of Lemma~\ref{lem=prel}, to show that $\Theta$ is an $m$-GRR for $G$ it suffices to prove that  $\Theta$ satisfies Lemma~\ref{lem=prel}~\eqref{enu3}. (Observe that this part of the proof holds also when $n\in \{6,7,8,9\}$ because we have verified the hypothesis~\eqref{enu2} in Lemma~\ref{lem=prel} with a computer.) From the definition of $\Theta$ and from our choice of $x$, we have $[\Theta(1_0)]\cong [\G(1_0)]\uplus \mathbf{K}_1$, $[\Theta(1_1)]\cong [\G(1_1)]\uplus \mathbf{K}_1$ and, for $i\in \{2,\ldots,m-1\}$, $[\Theta(1_i)]$ is isomorphic to the path $(a^{-2}_{i},a^{-1}_{i},a_{i},a^2_i)$ together with two isolated vertices, see Figure~\ref{Fig33}. Thus $[\Theta(1_0)]\ncong [\Theta(1_i)]$ and $[\Theta(1_1)]\ncong [\Theta(1_i)]$.

\begin{figure}[htbp]
\begin{center}
\unitlength 4mm
\begin{picture}(-10,17)
\put(-15, 13.2){\line(0, 1){2}}
\put(-15, 11.2){\line(0, 1){2}}
\put(-15,9.2){\line(0,1){2}}
\put(-19.2, 15){$\bullet$}
\put(-11.2, 15){$\bullet$}
\put(-15.2, 15){$\bullet$}
\put(-15.2, 13){$\bullet$}
\put(-15.2, 11){$\bullet$}
\put(-15.2,9){$\bullet$}

\put(-21, 15){$1_{i-1}$}
\put(-10.2,15){$1_{i+1}$}
\put(-14.2, 15){$a^2_i$}
\put(-14.2, 13){$a_i$}
\put(-14.2, 11){$a^{-1}_i$}
\put(-14.2,9){$a^{-2}_i$}
\put(-16.7, 7.5){$[\Theta(1_i)]$ with $2\leq i\leq m-2$}

\put(5.2, 13.2){\line(0, 1){2}}
\put(5.2, 11.2){\line(0, 1){2}}
\put(5.2,9.2){\line(0,1){2}}
\put(9, 15){$\bullet$}
\put(0.6, 15){$\bullet$}
\put(5, 15){$\bullet$}
\put(5, 13){$\bullet$}
\put(5, 11){$\bullet$}
\put(5,9){$\bullet$}

\put(-.8, 15){$x^{-1}_0$}
\put(10,15){$1_{m-2}$}
\put(6, 15){$a^2_{m-1}$}
\put(6, 13){$a_{m-1}$}
\put(6, 11){$a^{-1}_{m-1}$}
\put(6,9){$a^{-2}_{m-1}$}
\put(3.5, 7.5){$[\Theta(1_{m-1})]$}
\end{picture}
\end{center}\vspace{-3cm}
\caption{The induced subgraph $[\Theta(1_i)]$}\label{Fig33}
\end{figure}
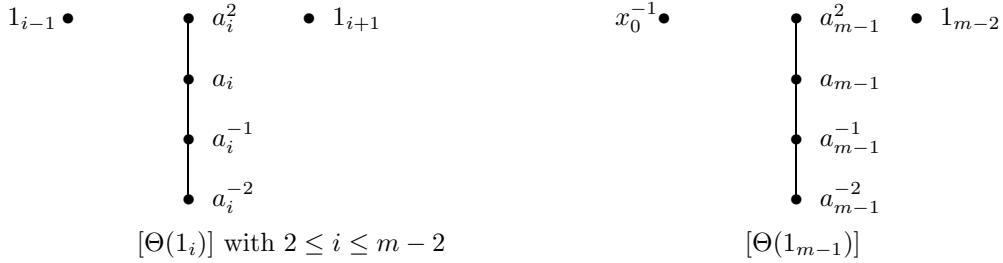

\smallskip

Suppose next that $G:=\langle a,b\rangle$ is dicyclic and $G\ncong Q_8$.
Then $o(a)>4$ is even.  Let $R:=\{a,a^{-1}\}$, $L:=\{a^2,a^{-2}\}$, $S:=\{1,a,a^3,b\}$, $T:=\{a,a^{-1},a^{o(a)/2},a^2,a^{-2}\}$ and $x:=a^2$. Observe that $R=R^{-1}$, $L=L^{-1}$, $T=T^{-1}$, $|R|=|L|=|T|-|S|+1$ and $x\notin S$.
Let $\Theta:=\Theta^m(G,R,L,S,T,x)$, $\G:=\BiCay(G,R,L,S)$ and $A:=\Aut(\G)$.

When $o(a)\in \{6,8\}$, a computation with \texttt{magma} shows that the hypothesis~\eqref{enu2} and~\eqref{enu3} in Lemma~\ref{lem=prel} are satisfied with $m=3$. Therefore, combining Lemmas~\ref{lem=prel} and~\ref{lem=prel2}, we obtain that $G$ admits an $m$-GRR, for every $m\ge 2$. Thus, for the rest of our argument, we may assume $o(a)\ge 10$. We have
\begin{align*}
\G_2(1_0)=\{&a_0^{-3},a_0^{-2},a_0^2,a_0^3,a_1^{-2},a_1^{-1},a_1^2,a_1^4,a_1^5,  b_0^{-1}, (ba^{-1})_1, (ba)_1,(b^{-1}a)_0,(b^{-1}a^3)_0,(a^{-2}b)_1,(a^{-3}b)_0,b_0,(a^{-1}b)_0\},(a^2b)_1\},\\
\G_2(1_1)=\{&a_0^{-5},a_0^{-4},a_0^{-2},a_0,a_0^2,a_1^{-4},a_1^{-3},
a_1^{-1},a_1,a_1^3,
  a_1^4,(b^{-1}a^{-2})_0, (ba^{-3})_1,b_1,  (ba^{-1})_1,
 (b^{-1}a^2)_0, (ab^{-1})_0,(a^{-1}b^{-1})_0,\\&b^{-1}_1,(ab^{-1})_1,(a^3b^{-1})_1\}.
\end{align*}
Since $|\G_2(1_0)|=19$ and $|\G_2(1_1)|=21$, $A$ has two orbits on $V\Gamma$, that is, $G_0$ and $G_1$.
Moreover, the induced subgraphs $[\G(1_0)]$ and $[\G(1_1)]$ are isomorphic to the graphs in Figure~\ref{Fig} together with (respectively) the isolated vertices $b_1$ and $(b^{-1})_0$. Since $A$ fixes setwise $G_0$ and $G_1$, $A_{1_0}$ fixes setwise $\Gamma(1_0)\cap G_0$ and $\G(1_0)\cap G_1$ and $A_{1_1}$ fixes  setwise $\Gamma(1_1)\cap G_0$ and $\G(1_1)\cap G_1$; thus $A_{1_0}$ fixes pointwise $\G(1_0)$ and $A_{1_1}$ fixes pointwise $\G(1_1)$. As $\G$ is connected, an easy connectedness argument implies that $A_{1_0}=A_{1_1}=1$, that is, $A=G$ and $\G$ is a $2$-GRR.

From our choice of $T$, we have $[\Theta(1_0)]\ncong [\Theta(1_i)]\ncong [\Theta(1_1)]$ with $i\in\{2,\cdots,m-1\}$ and hence, by Lemma~\ref{lem=prel}, $\Theta$ is an $m$-GRR for $G$.
\end{proof}

Lemma~\ref{lem=cyclic} deals already with most cyclic groups. We now give some ad-hoc constructions for the remaining small cases and for $Q_8$. Let $n$ and $m$ be positive integers with $n\ge 3$ and $m\ge 5$, let $\delta\in \{0,\ldots,n-1\}$ with $\gcd(1+\delta,n)=1$ and let $G$ be a cyclic group of order $n$ generated by $x$. We construct a graph $\Delta$ with $nm$ vertices: the vertex set $V\Delta$ of $\Delta$ is the Cartesian product $G\times \{0,\ldots,m-1\}$. Therefore, the vertex set $V\Delta$ is partitioned into $m$ subsets of cardinality $n$, namely $G\times \{i\}=G_i$ for $i\in \{0,\ldots,m-1\}$, which we call \textit{blocks}.
We now define the edges of $\Delta$, and for this it might be of some help looking at Figure~\ref{Fig1}.
\begin{enumerate}
\item The graph induced by $\Delta$ on $G_0$ and on $G_3$ is the empty graph, that is, $\Delta$ has no edges within the sets $G_0$ and $G_3$.
\item The graphs $[G_1]$ and  $[G_2]$ are complete graphs.
\item For each $i\in \{4,\ldots,m-1\}$, the edge set of the graph $[G_i]$ is $\{\{g_i,(xg)_i\}\mid g\in G\}$. In particular, $[G_i]$ is a cycle of length $o(x)=|G|=n$.
\end{enumerate}
Next, we define the edges between two distinct blocks of $\Delta$, these definitions are all natural, except for the edges between $G_0$ and $G_1$, and between $G_1$ and $G_2$. As usual, in defining the new edges  it might be of some help considering Figure~\ref{Fig1}:
\begin{enumerate}
\item edges between $G_0$ and $G_2$: for each $g,g'\in G$, the vertex $g_0$ is adjacent to $g_2'$ if and only if $g\ne g'$;
\item edges between $G_2$ and $G_3$: for each $g,g'\in G$, the vertex $g_2$ is adjacent to $g_3'$ if and only if $g= g'$;
\item edges between $G_3$ and $G_{m-1}$: for each $g,g'\in G$, the vertex $g_3$ is adjacent to $g_{m-1}'$ if and only if $g\ne g'$;
\item edges between $G_1$ and $G_4$: for each $g,g'\in G$, the vertex $g_1$ is adjacent to $g_4'$ if and only if $g\ne g'$;
\item edges between $G_\ell$ and $G_{\ell+1}$ for $\ell\in \{4,\ldots,m-2\}$: for each $g,g'\in G$, the vertex $g_\ell$ is adjacent to $g_{\ell+1}'$ if and only if $g\ne g'$;
\item edges between $G_0$ and $G_3$: for each $g,g'\in G$, the vertex $g_0$ is adjacent to $g_3'$;
\item edges between $G_0$ and $G_1$: for each $g,g'\in G$, the vertex $g_0$ is adjacent to $g_1'$ if and only if $g'g^{-1}=x$;
\item edges between $G_1$ and $G_2$: for each $g,g'\in G$, the vertex $g_1$ is adjacent to $g_2'$ if and only if $g'g^{-1}=x^\delta$.
\end{enumerate}

Thus $\Delta$ is a regular graph with $mn$ vertices and valency $2n$. In what follows, we identify $G$ with its image in $\mathrm{Sym}(G\times \{0,\ldots,m-1\})$ via its natural (component-wise) regular action, that is, for each $g\in G$, we identify $g$ with the permutation mapping $y_i$ to $(yg)_i$, for each $y_i\in G\times\{0,\ldots,m-1\}$.
\begin{figure}[!hhh]
\begin{center}
\begin{tikzpicture}[node distance=1.4cm,thick,scale=0.7,every node/.style={transform shape}]
\node[circle](AAAA0){};
\node[left=of AAAA0,circle](AAA0){};
\node[left=of AAA0,circle](AA0){};
\node[right=of AA0,circle,draw,inner sep=9pt, label=45:$G_3$](A0){$0$};
\node[left=of AA0,circle,draw,inner sep=5pt, label=45:$G_2$](A1){$n-1$};
\node[below=of A1,circle,draw,inner sep=5pt,label=-90:$G_1$](A2){$n-1$};
\node[left=of A2,circle,draw, inner sep=9pt, label=-90:$G_0$](A3){$0$};
\node[right=of A2,circle,draw, inner sep=10pt, label=-90:$G_4$](A4){$2$};
\node[right=of A4,circle,draw, inner sep=10pt, label=-90:$G_5$](A5){$2$};
\node[right=of A5,circle, inner sep=10pt](A6){};
\node[right=of A6,circle, inner sep=10pt](A7){};
\node[right=of A7,circle,draw, inner sep=9pt, label=-90:$G_{m-1}$](A8){$2$};
\draw(A1) to node[above]{$1$}  (A0);
\draw (A2) to node[above]{$1$} (A3);
\draw (A3) to node [left]{$n-1$}(A1);
\draw (A2) to node[above]{$n-1$} (A4);
\draw(A4) to node[above]{$n-1$} (A5);
\draw(A5) to node [above]{$n-1$}(A6);
\draw(A7) to node [above]{$n-1$}(A8);
\draw(A0) to [bend right=99] node [above]{$n$} (A3);
\draw (A1) to node[left]{$1$}(A2);
\draw (A0) to [bend left] node [above]{$n-1$}(A8);
\draw[dashed] (A6) to (A7);
\end{tikzpicture}
\end{center}
\caption{The graph $\Delta$~}\label{Fig1}
\end{figure}
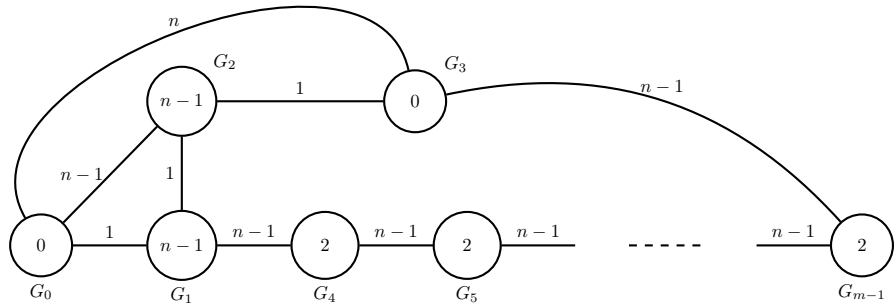

It is easy to see that $G$ acts as a group of automorphisms on the graph $\Delta$ and hence $G\le\mathrm{Aut}(\Delta)$. Moreover, the orbits of $G$ on the vertices of $\Delta$ are exactly the blocks $G_0,G_1,G_2,\ldots,G_{m-1}$.

\begin{lem}\label{l:1}The automorphism group of $\Delta$ is $G$. Therefore, for every cyclic group $G$ of order at least $3$ and for every $m\ge 5$, $G$ has an  $m$-$\mathrm{GRR}$.
\end{lem}
\begin{proof}
Let $A:=\mathrm{Aut}(\Delta)$. Using the definition of the edge set of $\Delta$, it is not difficult (but it does require some care) to show that
\begin{itemize}
\item $\Delta(1_0)=G_3\cup \{x_1\}\cup \{g_2\mid g\in G\setminus\{1\}\}$ and $[\Delta(1_0)]$ has $\frac{n^2-n+2}{2}$ edges;
\item $\Delta(1_1)=\{x_0^{-1}\}\cup\{g_1\mid g\in G\setminus\{1\}\}\cup \{x_2^\delta\}\cup \{g_4\mid g\in G\setminus\{1\}\}$ and $[\Delta(1_1)]$ has $\frac{3n^2-7n+4}{2}$ edges;
\item $\Delta(1_2)=\{g_0\mid g\in G\setminus\{1\}\}\cup \{x_1^{-\delta}\}\cup \{g_2\mid g\in G\setminus\{1\}\}\cup\{1_3\}$ and $[\Delta(1_2)]$ has $\frac{3n^2-7n+8}{2}$ edges;
\item $\Delta(1_3)=G_0\cup \{1_2\}\cup \{g_{m-1}\mid g\in G\setminus\{1\}\}$ and $[\Delta(1_3)]$ has $2n-3$ edges;
\item
$$\Delta(1_4)=\begin{cases}
\{g_1\mid g\in G\setminus\{1\}\}\cup \{x_4,x_4^{-1}\}\cup \{g_5\mid g\in G\setminus\{1\}\}&\textrm{ when }m\ge 6,\\
\{g_1\mid g\in G\setminus\{1\}\}\cup \{x_4,x_4^{-1}\}\cup \{g_3\mid g\in G\setminus\{1\}\}&\textrm{ when }m= 5,\\
\end{cases}$$
and the number of edges of $[\Delta(1_4)]$ is
$$
\begin{cases}
\frac{n^2+7n-16}{2}&\textrm{when }n=3 \textrm{ and }m\ge 6,\\
\frac{n^2+7n-18}{2}&\textrm{when }n>3 \textrm{ and }m\ge 6,\\
\frac{n^2+5n-12}{2}&\textrm{when }n=3 \textrm{ and }m=5,\\
\frac{n^2+5n-14}{2}&\textrm{when }n>3 \textrm{ and }m=5,\\
\end{cases}
$$
observing that the case $n=3$ is special here because $x_4$ and $x_4^{-1}$ are adjacent in $\Delta$ if and only if $x^3=1$, that is, $n=o(x)=3$;
\item for each $i\in \{5,\ldots,m-2\}$,
$\Delta(1_i)=\{g_{i-1}\mid g\in G\setminus\{1\}\}\cup\{x_i,x_i^{-1}\}\cup\{g_{i+1}\mid g\in G\setminus \{1\}\}$ and the number of edges of $[\Delta(1_i)]$ is
$$
\begin{cases}
6n-11&\textrm{when }n=3,\\
6n-12&\textrm{when }n>3,\\
\end{cases}
$$
observing that this case arises only when $m- 2\ge 5$, that is, $m\ge 7$, and  the case $n=3$ is again special here because $x_i$ and $x_i^{-1}$ are adjacent in $\Delta$ if and only if $x^3=1$, that is, $n=o(x)=3$;
\item
$$\Delta(1_{m-1})=\begin{cases}
\{g_{m-2}\mid g\in G\setminus\{1\}\}\cup \{x_{m-1},x_{m-1}^{-1}\}\cup \{g_{3}\mid g\in G\setminus\{1\}\}&\textrm{ when }m\ge 6,\\
\{g_1\mid g\in G\setminus\{1\}\}\cup \{x_4,x_4^{-1}\}\cup \{g_3\mid g\in G\setminus\{1\}\}&\textrm{ when }m= 5,\\
\end{cases}$$
and the number of edges of $[\Delta(1_{m-1})]$ is
$$
\begin{cases}
5n-9&\textrm{when }n=3 \textrm{ and }m\ge 6,\\
5n-10&\textrm{when }n>3 \textrm{ and }m\ge 6,\\
\frac{n^2+5n-12}{2}&\textrm{when }n=3 \textrm{ and }m=5,\\
\frac{n^2+5n-14}{2}&\textrm{when }n>3 \textrm{ and }m=5,\\
\end{cases}
$$
as above, the case $n=3$ being special because $x_{m-1}$ and $x_{m-1}^{-1}$ are adjacent in $\Delta$ if and only if $n=o(x)=3$.
\end{itemize}

Using the formulas above, we see that, for $g_i,g_j'\in V\Delta$, the number of edges of $[\Delta(g_i)]$ equals the number of edges of $[\Delta(g_j')]$ if and only if either
\begin{itemize}
\item $i=j$, or
\item $5\le i,j\le m-2$.
\end{itemize}
Since $|\Delta(g_i)|=|\Delta(g_i^a)|$ for every $a\in A$, we deduce that
$A$ fixes setwise $G_0,G_1,G_2,G_3,G_4,G_{m-1}$ and $G_5\cup G_6\cup\cdots\cup G_{m-2}$. For each $j\in \{4,\ldots,m-3\}$, the edges between $G_j$ and $G_{j+1}$ form  the complement of a complete matching; from this and from an inductive argument, it follows that $A$ fixes setwise $G_{i}$, for each $i\in \{0,\ldots,m-1\}$.

We claim that $A$ acts faithfully on each block $G_i$. Arguing by contradiction, suppose that there exists $i\in \{0,\ldots,m-1\}$ and $a\in A\setminus\{1\}$ with $a$ fixing pointwise $G_i$. Observe that, apart from $G_0$ and $G_3$, between any two adjacent blocks in $\Delta$ there is either a perfect matching or the complement of a perfect matching. Therefore, with the above mentioned exceptions of $G_0$ and $G_3$, the automorphism $a$ fixes pointwise the blocks adjacent to $G_i$. From this, it follows with a connectedness argument that $a=1$, contradicting our choice of $a$.

Finally we prove that $A=G$. Let $a\in A$ with $1_0^a=1_0$. Now, $1_0$ has only one neighbor in $G_1$, namely $x_1$, and hence $a$ fixes $x_1$. Similarly, $x_1$ has only one neighbor in $G_2$, namely $(x^{1+\delta})_2$, and hence $a$ fixes $(x^{1+\delta})_2$. Since $(x^{1+\delta})_0$ is the only vertex in $G_0$ not adjacent to $(x^{1+\delta})_2$, we have that $a$ fixes $(x^{1+\delta})_0$. Arguing in a similar manner with the vertex $1_0$ replaced by $(x^{1+\delta})_0$, we obtain that $a$ fixes $(x^{2(1+\delta)})_0$ and, in general, $a$ fixes pointwise $\{(x^{(1+\delta)i})_0\mid i\in \{0,\ldots,n-1\}\}$. As $\gcd(1+\delta,n)=1$, we have $\{(x^{(1+\delta)i})_0\mid i\in \{0,\ldots,n-1\}\}=G_0$ and  hence $a=1$ by the previous paragraph.
\end{proof}

For background on Cartesian products of graphs, the reader is refereed to~\cite{Sabidussi}. The same reference contains the following result, which we state here only for finite graphs. (A graph $X$ is said to be {\em relatively prime} with respect to Cartesian multiplication if there are no non-trivial graphs $Y$ and $Z$ with $X\cong Y\times Z$.)

\begin{lem}\label{lem=2.1}{{\cite[Corollary~$3.2$]{Sabidussi}}}
If $X_1$ and $X_2$ are connected graphs which are relatively prime with respect to Cartesian multiplication, then $\Aut(X_1\times X_2)=\Aut(X_1)\times \Aut(X_2)$.
\end{lem}
If $X$ is a graph, then its complement is denoted by $X^c$. We require the main result from~\cite{Imrich1}.
\begin{lem}\label{lem=2.2}{{~\cite[Theorem~$1$]{Imrich1}}}
If $X$ is a finite graph, then either $X$ or $X^c$ is prime with respect to Cartesian multiplication unless $X$ is one of the following six graphs:
$$\mathbf{K}_2\times \mathbf{K}_2,\,\,
\mathbf{K}_2\times \mathbf{K}_2^c,\,\,
\mathbf{K}_2\times \mathbf{K}_2\times \mathbf{K}_2,\,\,
\mathbf{K}_4\times \mathbf{K}_2,\,\,
\mathbf{K}_3\times \mathbf{K}_3,\,\,
 \mathbf{K}_2\times \mathbf{K}_4^-,$$
where $\mathbf{K}_4^-$ is obtained from $\mathbf{K}_4$ by deleting an edge.
\end{lem}

\begin{lem}\label{lem=small} For cyclic groups of order less than $6$, we have
\begin{itemize}
\item [\rm(1)] The cyclic group of order $1$ has an $m$-$\mathrm{GRR}$ if and only if $m=1$ or $m\ge 10$.
\item [\rm(2)] The cyclic group of order $2$ has an $m$-$\mathrm{GRR}$ if and only if $m=1$ or $m\ge 5$.
\item [\rm(3)] The cyclic group of order $3$ has an $m$-$\mathrm{GRR}$ if and only if $m\ge 4$.
\item [\rm(4)] The cyclic group of order $4$ or $5$ has an $m$-$\mathrm{GRR}$ if and only if $m\ge 3$.
\end{itemize}
\end{lem}
\begin{proof}
Let $G$ be a cyclic group of order $n$. When $n=1$, $G$ admits an $m$-GRR if and only if there exists a regular asymmetric graph of order $m$. By a nice result of Baron and Imrich~\cite{BI}, we have $m\ge 10$: there exist $4$-regular asymmetric graphs of order $m$ for each $m\ge 10$, and a regular graph with fewer than $10$ vertices is not asymmetric unless $m=1$.

Suppose $n=2$. Let $\Delta$ be any asymmetric regular graph with at least $10$ vertices. By Lemma~\ref{lem=2.2}, replacing $\Delta$ by $\Delta^c$ if necessary, we may assume that $\Delta$ is prime with respect to Cartesian multiplication. Clearly, $\mathbf{K}_2$ is also prime with respect to Cartesian multiplication. Therefore, by Lemma~\ref{lem=2.1}, we have $\Aut(\Delta\times \mathbf{K}_2)=\Aut(\Delta)\times \Aut(\mathbf{K}_2)\cong\Aut(\mathbf{K}_2)$ and hence $\Aut(\Delta\times \mathbf{K}_2)$ is cyclic of order $2$. This shows that $\Delta\times \mathbf{K}_2$ is an $m$-GRR for $G$. Thus $G$ has an $m$-GRR for every $m\ge 10$. Clearly, $G=\Aut(\mathbf{K}_2)$ and hence $G$ has a $1$-GRR. A computation with \texttt{magma} shows that $G$ has no $m$-GRR when $m\in \{2,3,4\}$, and $G$ admits an $m$-GRR for each $m\in \{5,6,7,8,9\}$.

Suppose $n=3$. From Lemma~\ref{l:1}, $G$ has an $m$-GRR when $m\ge 5$. Now a computation with \texttt{magma} yields that $G$ has no $m$-GRRs when $m\in \{2,3\}$, but $G$ has a $4$-GRR.

Suppose $n\in \{4,5\}$. From Lemma~\ref{l:1}, $G$ has an $m$-GRR when $m\ge 5$. Now a computation with \texttt{magma} yields that $G$ has an $m$-GRR also when $m\in \{3,4\}$, but $G$ has no $2$-GRR.
\end{proof}

We now study $G:=Q_8=\langle i,j\mid i^4=j^4=j^{-1}iji=1,i^2=j^2\rangle$. Let $m$ be a positive integer with $m\ge 3$. We construct a graph $\Delta$ with $8m$ vertices of valency $5$: the vertex set $V\Delta$ of $\Delta$ is the Cartesian product $G\times\{0,\ldots,m-1\}=G_0\cup\cdots \cup G_{m-1}$. We now define the edges of $\Delta$:
\begin{enumerate}
\item The graph induced by $\Delta$ on $G_0$ is $\Cay(G,\{i,i^2,i^3\})$; for each $\ell\in \{1,\ldots,m-3\}$, the graph induced by $\Delta$ on $G_\ell$ is $\Cay(G,\{i^2\})$ (observe that this case arises only when $m\ge 4$); the graph induced by $\Delta$ on $G_{m-2}$ is the empty graph; the graph induced by $\Delta$ on $G_{m-1}$ is $\Cay(G,\{j,j^{-1}\})$.
\item For each $\ell\in \{0,\ldots,m-3\}$ and for each $g\in G$, $g_\ell\in G_\ell$ is adjacent to $g_{\ell+1}\in G_{\ell+1}$ and to $(ig)_{\ell+1}\in G_{\ell+1}$.
\item For each $g\in G$, $g_{m-2}\in G_{m-2}$ is adjacent to $g_{m-1}\in G_{m-1}$, to $(ig)_{m-1}\in G_{m-1}$ and to $(jg)_{m-1}\in G_{m-1}$.
\end{enumerate}
\begin{lem}\label{lem=Q8}The automorphism group of $\Delta$ is $G=Q_8$. Therefore, for every $m\ge 3$, $Q_8$ admits an $m$-$\mathrm{GRR}$.
\end{lem}
\begin{proof}
Let $A:=\Aut(\Delta)$. It follows from a computation that $[\Delta(1_0)]\ncong [\Delta(g_\ell)]$, for every $g\in G$ and $\ell\in \{1,\ldots,m-1\}$, and that $[\Delta(1_{m-1})]\ncong [\Delta(g_\ell)]$, for every $g\in G$ and $\ell\in \{0,\ldots,m-2\}$. Therefore $A$ fixes setwise $G_0$ and $G_{m-1}$. Using the ``path-type" structure of $\Delta$, we deduce that $A$ fixes setwise $G_\ell$, for each $\ell\in \{0,\ldots,m-1\}$. We omit the rest of the proof, which rely on detailed computations on the local structure of $\Delta$.
\end{proof}

\begin{cor}\label{cor=cyclic}Let $G$ be a group as in Notation~$\ref{hyp1}$. Then $G$ admits an $m$-$\mathrm{GRR}$ unless one of the following holds:
\begin{enumerate}
\item $G$ is cyclic of order $1$ and $2\le m\le 9$;
\item $G$ is cyclic of order $2$ and $2\le m\le 4$;
\item $G$ is cyclic of order $3$ and $m\le 3$;
\item $G$ is cyclic of order $4$ or $5$ and $m\le 2$;
\item $G$ is cyclic of order at least $6$ and $m=1$;
\item $G=Q_8$ and $m\le 2$.
\end{enumerate}
Conversely, for each of the above $(m,G)$, $G$ has no $m$-$\mathrm{GRR}$.
\end{cor}

\begin{proof}
By Lemmas~\ref{lem=cyclic} and \ref{lem=small}, we have (1)-(5).
By Lemma~\ref{lem=Q8}, $Q_8$ has an $m$-GRR for each $m\geq 3$ and, by Proposition~\ref{prop=GRR}, $Q_8$ has no GRR. A computation with \texttt{magma} shows that
 $Q_8$ has no $2$-GRR.
\end{proof}

\subsection{Part 2: Abelian groups of rank $2$ and generalized dicyclic groups over them}
\begin{notation}\label{hyp2}{\rm We set some notation that we use in this section:
$G:=\langle a_1,a_2\rangle$ is an abelian group  with two generators such that $o(a_2)\mid o(a_1)$ and $o(a_2)>1$; or $G:=\langle a_1,a_2,b\rangle $ is a generalized dicyclic group over the abelian group  $\langle a_1,a_2\rangle$ of even order, exponent greater than two, with two generators such that $o(a_2)\mid o(a_1)$, $o(a_2)>1$, $a_1^b=a_1^{-1}$ and $a_2^b=a_2^{-1}$.}
\end{notation}

\begin{lem}\label{abe=1}Let $m$ be a positive integer with $m\ge 2$ and let $G$ be as in Notation~$\ref{hyp2}$ with $o(a_1)>2$. Then $G$ has an $m$-$\mathrm{GRR}$.
\end{lem}
\begin{proof}
Suppose first that $G:=\langle a_1\rangle\times \langle a_2\rangle$ is abelian with $o(a_2)\mid o(a_1)$ and $o(a_2)>1$. Assume $o(a_1)\in \{3,4\}$. Let
\begin{align*}
R:=&
\begin{cases}
\{a_1,a_1^{-1},a_2,a_2^{-1}\}&\textrm{if }o(a_2)\ne 2,\\
\{a_1,a_1^{-1},a_2\}&\textrm{if }o(a_2)=2,
\end{cases}\\
L:=&
\begin{cases}
\{a_1,a_1^{-1},a_1a_2,(a_1a_2)^{-1}\}&\textrm{if }o(a_2)\ne 2,\\
\{a_1,a_1^{-1},a_1^{o(a_1)/2}\}&\textrm{if }o(a_2)=2,
\end{cases}\\
S:=&
\{1,a_1,a_1a_2^{-1}\},\\
T:=&
\begin{cases}
\{a_1,a_1^{-1},a_2,a_2^{-1},a_1a_2,(a_1a_2)^{-1}\}&\textrm{if }o(a_2)\ne2,\\
\{a_1,a_1^{-1},a_2,a_1a_2,a_1^{-1}a_2\}&\textrm{if }o(a_2)=2,
\end{cases},\\
x:=&a_1^{-1}.
\end{align*}
Observe that $R=R^{-1}$, $L=L^{-1}$, $T=T^{-1}$, $|R|=|L|=|T|-|S|+1$ and $x\notin S$.
Let $\Theta:=\Theta^m(G,R,L,S,T,x)$, $\G:=\BiCay(G,R,L,S)$ and $A:=\Aut(\G)$. With a computer-aided computation (see Lemma~\ref{lem=prel2}), we get that $\G$ is a $2$-GRR for $G$ and that, for each $j\in \{2,\ldots,m-1\}$, $[\Theta(1_0)]\ncong [\Theta(1_{j})]$ and $[\Theta(1_1)]\ncong [\Theta(1_j)]$. Therefore, by Lemma~\ref{lem=prel}, $\Theta$ is an $m$-GRR for $G$ for every $m\ge 3$.

Assume $o(a_1)\geq 5$. Set
\begin{align*}
R&:=\{a_1,a_1^{-1},a_1a_2,(a_1a_2)^{-1} \},\\
L&:=\{a_1,a_1^{-1},a_1^2,a_1^{-2}\},\\
S&:=\{1,a_1,a_1a_2\},\\
T&:=\{a_1,a_1^{-1},a_1^2,a_1^{-2},a_1a_2,(a_1a_2)^{-1}\},\\
x&:=a_1^{-1}.
\end{align*}
Observe that $R=R^{-1}$, $L=L^{-1}$, $T=T^{-1}$, $|R|=|L|=|T|-|S|+1$ and $x\notin S$.
Let $\Theta:=\Theta^m(G,R,L,S,T,x)$, $\G:=\BiCay(G,R,L,S)$ and $A:=\Aut(\G)$. We start by proving that $\G$ is a $2$-GRR for $G$. Since $G=\lg R,L,S\rg$, $\G$ is connected. The graphs $[\G(1_0)]$ and $[\G(1_1)]$ are drawn in Figure~\ref{Fig3_33}.

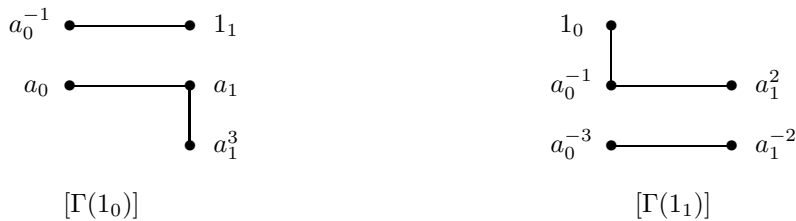
\begin{figure}[!hhh]
\begin{center}
\begin{tikzpicture}[node distance=1cm,thick,scale=1,every node/.style={transform shape}]
\node[circle,inner sep=1pt,draw, label=-180:$(a_1a_2)^{-1}_0$](A0){};
\node[right=of A0](AA0){};
\node[right=of AA0, circle,draw, inner sep=1pt, label=right:$(a_1a_2)_1$](B0){};
\node[above=of A0, circle,draw, inner sep=1pt, label=left:$(a_1^{-1})_0$](A1){};
\node[above=of A1, circle,draw, inner sep=1pt, label=left:$(a_1)_0$](A2){};
\node[above=of B0, circle,draw, inner sep=1pt, label=right:$1_1$](B1){};
\node[above=of B1, circle,draw, inner sep=1pt, label=right:$(a_1)_1$](B2){};
\node[below=of A0, circle,draw, inner sep=1pt, label=left:$(a_1a_2)_0$](A3){};

\draw (A0) to  (B1);
\draw (A1) to (B1);
\draw (A2) to (B2);
\draw (A3) to (B0);
\draw (B1) to (B2);

\node[below=of A5,circle, inner sep=1pt,label=above:](A6){};

\node[right=of B2](BBB2){};
\node[right=of BBB2](BB2){};
\node[right=of BB2, circle,draw, inner sep=1pt, label=left:$(a_1a_2)^{-1}_0$](AA0){};
\node[right=of AA0](BB){};
\node[right=of BB,circle,draw, inner sep=1pt, label=right:$(a_1^2)_1$](BB0){};
\node[below=of AA0, circle,draw, inner sep=1pt, label=left:$1_0$](AA1){};
\node[below=of AA1, circle,draw, inner sep=1pt, label=left:$(a_1^{-1})_0$](AA2){};

\node[below=of BB0, circle,draw, inner sep=1pt, label=right:$(a_1)_1$](BB1){};
\node[below=of BB1, circle,draw, inner sep=1pt, label=right:$(a_1^{-1})_1$](BB2){};
\node[below=of BB2, circle,draw, inner sep=1pt, label=right:$(a_1^{-2})_1$](BB3){};

\draw (AA0) to (AA1);
\draw (AA1) to (AA2);
\draw (AA1) to (BB1);
\draw (AA2) to (BB2);

\draw (BB0) to (BB1);
\draw (BB1) to (BB2);
\draw (BB2) to (BB3);
\node[below=of BB3,circle, inner sep=1pt,label=below:](AA6){};
\end{tikzpicture}
\end{center}\vskip-1.8cm
\caption{The induced subgraphs $[\G(1_0)]$ and $[\G(1_1)]$}\label{Fig3_33}
\end{figure}

Since $|E[\G(1_0)]|=5$ and $|E[\G(1_1)]|=7$, $A$ has
two orbits on $V\G$, that is, $G_0$ and $G_1$.
Since $A$ fixes  setwise $G_0$ and $G_1$, $A_{1_0}$ fixes setwise $\Gamma(1_0)\cap G_0$ and $\G(1_0)\cap G_1$ and $A_{1_1}$ fixes  setwise $\Gamma(1_1)\cap G_0$ and $\G(1_1)\cap G_1$. Now, by looking at  Figure~\ref{Fig3_33}, we deduce that  $A_{1_0}$ fixes pointwise $\G(1_0)$ and $A_{1_1}$ fixes pointwise $\G(1_1)$. As $\G$ is connected, an easy connectedness argument implies that $A_{1_0}=A_{1_1}=1$, that is, $A=G$ and $\G$ is a $2$-GRR.

Now, in the light of Lemma~\ref{lem=prel}, it suffices to show that the graph $\Theta$ satisfies Lemma~\ref{lem=prel}~\eqref{enu3}. From the definition of $\Theta$ and from our choice of $x$, we have $[\Theta(1_0)]\cong [\G(1_0)]\uplus \mathbf{K}_1$ and $[\Theta(1_1)]\cong [\G(1_1)]\uplus \mathbf{K}_1$; it is a simple computation to show that, for each $i\in \{2,\ldots,m-1\}$, $[\Theta(1_i)]$ is isomorphic to neither $[\Theta(1_0)]$ nor $[\Theta(1_1)]$.

\smallskip

Suppose next that $G:=\langle a_1,a_2,b\rangle$ is generalized dicyclic with $o(a_2)\mid o(a_1)$, $o(a_2)>1$, $a_1^b=a_1^{-1}$ and $a_2^{b}=a_2^{-1}$. Since $\langle a_1,a_2\rangle$ has even order, $o(a_1)$ cannot be odd. When $o(a_1)=4$, take $R$, $L$, and $x$   the same as in  the case of $G$ abelian, but add $b$ to $S$ and define
\[T:=\begin{cases}
\{a_1,a_1^{-1},a_2,a_2^{-1},a_1a_2,(a_1a_2)^{-1},b^2\}& \textrm{when }o(a_2)\ne 2,\\
\{a_1,a_1^{-1},a_1a_2,(a_1a_2)^{-1},b,b^{-1}\}& \textrm{when }o(a_2)=2.
\end{cases}\] Then a \texttt{magma} computation (together with Lemmas~\ref{lem=prel} and~\ref{lem=prel2}) shows that $\Theta^m(G,R,L,S,T,x)$ is an $m$-GRR for $G$. Therefore $G$ admits an $m$-GRR, for every $m\ge 2$.

When $o(a_1)\geq 5$, take $R$, $L$ and $x$  the same as  in the case of $G$ abelian, and define $\tilde{S}:=S\cup \{b\}$ and $\tilde{T}=T\cup\{b^2\}$. Observe that $R=R^{-1}$, $L=L^{-1}$, $\tilde{T}=\tilde{T}^{-1}$, $|R|=|L|=|\tilde{T}|-|\tilde{S}|+1$ and $x\notin \tilde{S}$.
Let $\tilde{\Theta}:=\Theta^m(G,R,L,\tilde{S},\tilde{T},x)$, $\tilde{\G}:=\BiCay(G,R,L,S)$ and $A:=\Aut(\tilde{\G})$. (Recall the definition of $\Theta$ and $\G$ in the abelian case above.)  We view $\G$ as a subgraph of $\tilde{\G}$ and $\Theta$ as a subgraph of $\tilde{\Theta}$.
Then $[\tilde{\G}(1_0)]\cong \mathbf{K}_1\uplus [\G(1_0)]$ and $[\tilde{\G}(1_1)]\cong\mathbf{K}_1\uplus[\G(1_1)]$; moreover, since $[\G(1_0)]\ncong[\G(1_1)]$, we have $[\tilde{\G}(1_0)]\ncong[\tilde{\G}(1_1)]$. Therefore,
$A$ has two orbits on $V\tilde{\G}$, that is, $G_0$ and $G_1$. Furthermore, since $[\tilde{\G}(1_0)]=\{b_1\}\cup [\G(1_0)]$ and since $[\G(1_0)]$ has no isolated vertices, $A_{1_0}$ fixes setwise $[\G(1_0)]$. Similarly, $A_{1_1}$ fixes setwise $[\G(1_1)]$. Therefore arguing as in the abelian case above or by direct inspection in Figure~\ref{Fig3_33}, we obtain that $A_{1_0}$ fixes pointwise $[\G(1_0)]$ and $A_{1_1}$ fixes pointwise $[\G(1_1)]$. Thus $A_{1_0}$ fixes pointwise $[\G(1_0)]\cup\{b_1\}=[\tilde{\G}(1_0)]$ and $A_{1_1}$ fixes pointwise $[\G(1_1)]\cup\{(b^{-1})_0\}=[\tilde{\G}(1_1)]$.  Since $\tilde{\G}$ is connected, we have $A_{1_0}=A_{1_1}=1$, that is, $A=G$ and $\tilde{\G}$ is a $2$-GRR.

From our choice of $\tilde{T}$, we have $[\tilde{\Theta}(1_0)]\ncong [\tilde{\Theta}(1_i)]\ncong [\tilde{\Theta}(1_1)]$ for every $i\in \{2,\ldots,m-1\}$ and hence, by Lemma~\ref{lem=prel}, $\tilde{\Theta}$ is an $m$-GRR for $G$.
\end{proof}

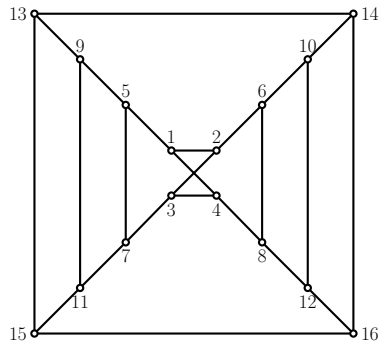
\begin{figure}[!ht]
\begin{center}
\begin{tikzpicture}[node distance=1.2cm,thick,scale=0.7,every node/.style={transform shape},scale=0.6]
\node[circle,inner sep=2pt,draw, label=above:{\huge$1$}](A0){};
\node[right=of A0, circle,draw, inner sep=2pt, label=above:{\huge$2$}](B0){};
\node[below=of A0, circle,draw, inner sep=2pt, label=below:{\huge$3$}](A1){};
\node[right=of A1, circle,draw, inner sep=2pt, label=below:{\huge$4$}](B1){};
\node[left=of A0](A2){};
\node[below=of A2](A3){};

\node[above=of A2, circle,draw, inner sep=2pt, label=above:{\huge$5$}](A20){};
\node[below=of A3, circle,draw, inner sep=2pt, label=below:{\huge$7$}](A21){};
\node[left=of A21](A4){};
\node[left=of A20](A5){};
\node[above=of A5, circle,draw, inner sep=2pt, label=above:{\huge$9$}](A30){};
\node[below=of A4, circle,draw, inner sep=2pt, label=below:{\huge$11$}](A31){};
\node[left=of A31](A6){};
\node[left=of A30](A7){};
\node[above=of A7, circle,draw, inner sep=2pt, label=left:{\huge$13$}](A40){};
\node[below=of A6, circle,draw, inner sep=2pt, label=left:{\huge$15$}](A41){};
\node[right=of B0](B2){};
\node[below=of B2](B3){};

\node[above=of B2, circle,draw, inner sep=2pt, label=above:{\huge$6$}](B20){};
\node[below=of B3, circle,draw, inner sep=2pt, label=below:{\huge$8$}](B21){};
\node[right=of B21](B4){};
\node[right=of B20](B5){};
\node[above=of B5, circle,draw, inner sep=2pt, label=above:{\huge$10$}](B30){};
\node[below=of B4, circle,draw, inner sep=2pt, label=below:{\huge$12$}](B31){};
\node[right=of B30](B6){};
\node[right=of B31](B7){};
\node[above=of B6, circle,draw, inner sep=2pt, label=right:{\huge$14$}](B40){};
\node[below=of B7, circle,draw, inner sep=2pt, label=right:{\huge$16$}](B41){};
\draw (A0) to  (B0);
\draw (A1) to (B1);
\draw (A21) to (A1);
\draw (A31) to (A21);
\draw (A41) to (A31);
\draw (A20) to (A0);
\draw (A30) to (A20);
\draw (A40) to (A30);
\draw (A1) to  (B0);
\draw (A0) to  (B1);
\draw (A20) to (A21);
\draw (A30) to (A31);
\draw (A40) to (A41);
\draw (B20) to (B0);
\draw (B30) to (B20);
\draw (B40) to (B30);
\draw (B21) to (B1);
\draw (B31) to (B21);
\draw (B41) to (B31);
\draw (B20) to (B21);
\draw (B30) to (B31);
\draw (B40) to (B41);
\draw (A40) to (B40);
\draw (A41) to (B41);
\end{tikzpicture}
\end{center}
\caption{A $4$-GRR for $\mathbb{Z}_2\times\mathbb{Z}_2$~}\label{Fig1_1_1}
\end{figure}

\begin{cor}\label{cor=abelian2}Let $G$ be as  in Notation~$\ref{hyp2}$. Then $G$ admits an $m$-$\mathrm{GRR}$ unless $m=1$, or $G\cong \mathbb{Z}_2\times \mathbb{Z}_2$ and $m= 2$.
\end{cor}
\begin{proof}
From Lemma~\ref{abe=1}, it suffices to consider the case $o(a_1)=2$, that is $G\cong\mathbb{Z}_2\times\mathbb{Z}_2$. A computation with \texttt{magma} and Proposition~\ref{prop=GRR} show that $G$ has no GRR and no $2$-GRR. Let $m\geq 3$.
Consider the cubic graph $\Sigma_m$ with
\begin{align*}
V:=&\{1,\ldots,4m\},\\
E:=&\{\{1,2\},\{2,3\},\{3,4\},\{1,4\}\}\cup\{\{x,x+4\}\mid x\in \{1,\ldots,4m-4\}\}\\
&\cup\{\{4\ell+1,4\ell+3\},\{4\ell+2,4\ell+4\}\mid \ell\in \{1,\ldots,m-1\}\\
&\cup\{\{4m-1,4m\},\{4m-2,4m-3\}\}.
\end{align*}
We have drawn this graph when $m=4$ in Figure~\ref{Fig1_1_1}.

Let $A:=\Aut(\Sigma_m)$. It is easy to see that, for any $4$-cycle $C$ with $C\ne (1,2,3,4)$, $C$ has an edge lying on another $4$-cycle. Thus $A$ fixes $\{1,2,3,4\}$
setwise. From this, it follows that  $A$ fixes $\{5,6,7,8\}$ setwise and, arguing inductively, $A$ fixes each layer $\{4i+1,4i+2,4i+3,4i+4\}$ setwise.

On the other hand, define $\a=\prod_{i=0}^{m-1}(4i+1,4i+2)(4i+3,4i+4)$ and $\b=\prod_{i=0}^{m-1}(4i+1,4i+3)(4i+2,4i+4)$. Then $\langle \a,\b\rangle\cong\mz_2^2$ is a semiregular subgroup of $A$.

Let $\ell=4s+1$ be a vertex of $\Sigma_m$. Then $A_\ell$ fixes $4(m-1)+1$ and  $4(m-2)+1$, because $A$ fixes each layer $\{4i+1,4i+2,4i+3,4i+4\}$ setwise. Furthermore, $A_\ell$ fixes $4(m-2)+3$ and $4(m-1)+3$, and hence $A_\ell$ fixes $\{4(m-1)+1,4(m-1)+2,4(m-1)+3,4(m-1)+4\}$ pointwise. This implies that $A_\ell=1$, and $A$ is semiregular because $A$ is transitive on  $\{4i+1,4i+2,4i+3,4i+4\}$. It follows that $A=\langle \a,\b\rangle$, which is semiregular with $m$-orbits.
\end{proof}

\subsection{Part 3: Abelian groups of rank at least 3 and the generalized dicyclic groups over them}
\begin{notation}\label{hyp3}{\rm We set some notation that we use in this section: $\kappa$ is a positive integer with $\kappa\ge 3$,
$G:=\langle a_1,\ldots,a_\kappa\rangle$ is an abelian group  with $\kappa$ generators such that $o(a_{i})\mid o(a_{i-1})$ for every $i\in \{2,\ldots,\kappa\}$ and $o(a_\kappa)>1$; or $G:=\langle a_1,\ldots,a_\kappa,b\rangle $ is a generalized dicyclic group over the abelian group  $\langle a_1,\ldots,a_\kappa\rangle$ of even order, exponent greater than two, with $\kappa$ generators such that $o(a_i)\mid o(a_{i-1})$ for every $i\in \{2,\ldots,\kappa\}$, $o(a_\kappa)>1$, $a_i^b=a_i^{-1}$ for every $i\in \{1,\ldots,\kappa\}$. We denote by $\ell\in \{1,\ldots,\kappa\}$ the largest integer with  $o(a_i)>2$ for each $i\in \{1,\ldots,\ell\}$ and $o(a_{\ell+1})=2$. Observe that $\ell$ is well-defined except when $G=\langle a_1,\ldots,a_\kappa\rangle$ has exponent $2$; in this case, we set $\ell:=0$.}
\end{notation}

\begin{lem}\label{abe=3}Let $m$ be a positive integer with $m\ge 2$ and let $G$ be as in Notation~$\ref{hyp3}$ with $o(a_1)>2$. Then $G$ has an $m$-$\mathrm{GRR}$.
\end{lem}
\begin{proof}
Suppose first that $G$ is abelian. We use the notation established in Notation~\ref{hyp3}. Let
\begin{align*}
R:=&\{a_1,a_2,\ldots,a_{\kappa}\} \cup\{a_1,a_2,\ldots,a_{\kappa}\}^{-1},\\
L:=&\{a_1a_2,a_2a_3,\ldots,a_{\kappa-1}a_\kappa\}\cup \{a_1a_2,a_2a_3,\ldots,a_{\kappa-1}a_\kappa\}^{-1}\cup\{a_\kappa,a_\kappa^{-1}\};\\
S:=&\{1,a_1,a_1^{-1},a_1a_2,a_\kappa\},\\
T:=&R\cup\{a_1a_2,(a_1a_2)^{-1},a_1a_3,(a_1a_3)^{-1}\},\\
x:=&(a_1a_2)^{-1}.
\end{align*}
Since $o(a_ia_{i+1})=o(a_i)$ for each $i\in \{1,\ldots,\kappa-1\}$, we have $|R|=|L|=\ell+\kappa$, $|S|=5$, $|T|=\ell+\kappa+4$ and $|R|=|L|=|T|-|S|+1$. Let $\Theta:=\Theta^m(G,R,L,S,T,x)$, $\G=\BiCay(G,R,L,S)$ and $A:=\Aut(\G)$. For any subset $Z\subseteq G$, write $Z_0=\{z_0\ |\ z\in Z\}$ and $Z_1=\{z_1\ |\ z\in Z\}$.

We start by proving that $\Gamma$ is a $2$-GRR. Observe that
\begin{align*}
\Gamma(1_0)&=R_0\cup\{1_1,(a_1)_1,(a_1^{-1})_1,(a_1a_2)_1,(a_\kappa)_1\},\quad \Gamma(1_1):=L_1\cup\{1_0,(a_1^{-1})_0,(a_1)_0,((a_1a_2)^{-1})_0,(a_\kappa^{-1})_0\}.
\end{align*}
Since $\{a_1,\ldots,a_\kappa\}$ is a minimal generating set for the abelian group $G$, so is $\{a_1a_2,\ldots,a_{\kappa-1}a_\kappa,a_\kappa\}$. Therefore, the only edges in the subgraph $\Gamma[R_0]$ are between $(a_i)_0$ and $(a_i^{-1})_0$, when $o(a_i)=3$, and the only edges in the subgraph $\G[L_1]$  are  between $(a_ia_{i+1})_1$ and $((a_ia_{i+1})^{-1})_1$ when $o(a_i)=3$ and between $(a_\kappa)_1$ and $(a_\kappa^{-1})_1$ when $o(a_\kappa)=3$. Let us denote by $I$ the subset of $\{a_1,a_2,\ldots,a_\kappa\}$ consisting of the elements having order $3$. Thus $\G[R_0]$ and $\G[L_1]$ both consist of $|I|$  parallel edges.

For $X\in \{R,L,S\}$, we let $E([\G(1_0)])_X$ denote the
set of $X$-edges in the induced subgraph $[\G(1_0)]$
of $\G(1_0)$ in $\G$, where $\G(1_0)$ is the neighborhood of $1_0$ in $\G$, and similarly, we have the notation
$E([\G(1_1)])_X$.
Then
\begin{align*}
E([\G(1_0)])&=E([\G(1_0)])_R+E([\G(1_0)])_L+E([\G(1_0)])_S,\\
E([\G(1_1)])&=E([\G(1_1)])_R+E([\G(1_1)])_L+E([\G(1_1)])_S.
\end{align*}
Clearly, $R$-edges connect vertices in $G_0$, $L$-edges connect vertices in $G_1$, and $S$-edges connect vertices between $G_0$ and $G_1$.
 Then $\G(1_0)=R_0\cup S_1$ and $\G(1_1)=L_1\cup S^{-1}_0$.
Since $|S|=5$, it is easy to determine $E([\G(1_0)])_L$, $E([\G(1_0)])_S$,
$E([\G(1_1)])_R$ and $E([\G(1_1)])_S$, and we have drawn these edges in Figure~\ref{Fig2abra}.
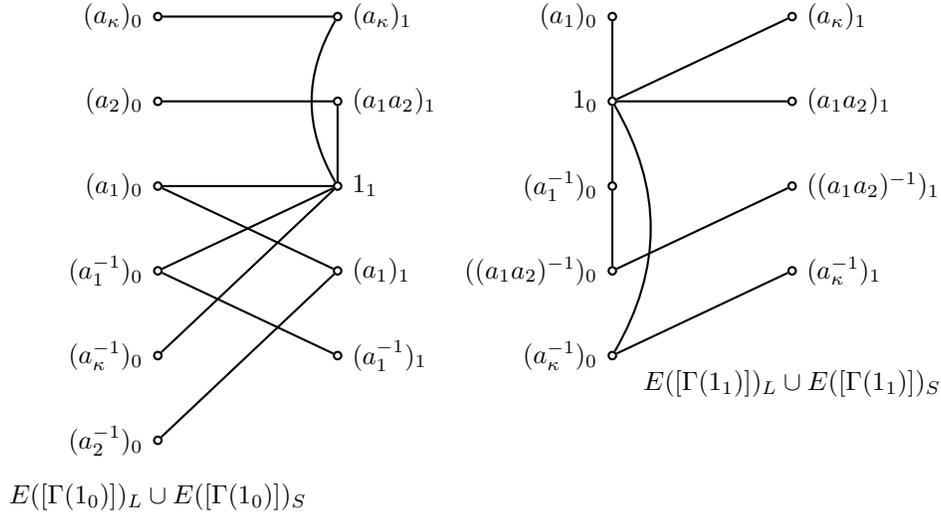
\begin{figure}[!hhh]
\begin{center}
\begin{tikzpicture}[node distance=1cm,thick,scale=1,every node/.style={transform shape}]
\node[circle,inner sep=1pt,draw, label=-180:$(a_1)_0$](A0){};
\node[right=of A0](AA0){};
\node[right=of AA0, circle,draw, inner sep=1pt, label=right:$1_1$](B0){};
\node[above=of A0, circle,draw, inner sep=1pt, label=left:$(a_2)_0$](A1){};
\node[above=of A1, circle,draw, inner sep=1pt, label=left:$(a_\kappa)_0$](A2){};
\node[above=of B0, circle,draw, inner sep=1pt, label=right:$(a_1a_2)_1$](B1){};
\node[above=of B1, circle,draw, inner sep=1pt, label=right:$(a_\kappa)_1$](B2){};
\node[below=of A0, circle,draw, inner sep=1pt, label=left:$(a_1^{-1})_0$](A3){};
\node[below=of A3, circle,draw, inner sep=1pt, label=left:$(a_\kappa^{-1})_0$](A4){};
\node[below=of B0, circle,draw, inner sep=1pt, label=right:$(a_1)_1$](B3){};
\node[below=of B3, circle,draw, inner sep=1pt, label=right:$(a_1^{-1})_1$](B4){};
\node[below=of A4, circle,draw, inner sep=1pt, label=left:$(a_2^{-1})_0$](A5){};
\draw (A0) to  (B0);
\draw (A1) to (B1);
\draw (A2) to (B2);
\draw (B0) to (B1);
\draw [bend left](B0) to (B2);
\draw (A0) to (B3);
\draw (A3) to (B0);
\draw (A3) to (B4);
\draw (A4) to (B0);
\draw (A5) to (B3);
\node[below=of A5,circle, inner sep=1pt,label=above:$E({[}\G(1_0){]})_L\cup E({[}\G(1_0){]})_S$](A6){};
\node[right=of B2](BBB2){};
\node[right=of BBB2](BB2){};
\node[right=of BB2, circle,draw, inner sep=1pt, label=left:$(a_1)_0$](AA0){};
\node[right=of AA0](BB){};
\node[right=of BB,circle,draw, inner sep=1pt, label=right:$(a_\kappa)_1$](BB0){};
\node[below=of AA0, circle,draw, inner sep=1pt, label=left:$1_0$](AA1){};
\node[below=of AA1, circle,draw, inner sep=1pt, label=left:$(a_1^{-1})_0$](AA2){};
\node[below=of AA2, circle,draw, inner sep=1pt, label=left:$((a_1a_2)^{-1})_0$](AA3){};
\node[below=of AA3, circle,draw, inner sep=1pt, label=left:$(a_\kappa^{-1})_0$](AA4){};
\node[below=of BB0, circle,draw, inner sep=1pt, label=right:$(a_1a_2)_1$](BB1){};
\node[below=of BB1, circle,draw, inner sep=1pt, label=right:$((a_1a_2)^{-1})_1$](BB2){};
\node[below=of BB2, circle,draw, inner sep=1pt, label=right:$(a_\kappa^{-1})_1$](BB3){};
\draw (AA0) to (AA1);
\draw (AA1) to (AA2);
\draw (AA2) to (AA3);
\draw [bend left] (AA1) to (AA4);
\draw (AA1) to (BB0);
\draw (AA1) to (BB1);
\draw (AA3) to (BB2);
\draw (AA4) to (BB3);
\node[below=of BB3,circle, inner sep=1pt,label=below:$E({[}\G(1_1){]})_L\cup E({[}\G(1_1){]})_S$](AA6){};
\end{tikzpicture}
\end{center}
\caption{The subgraphs $E([\G(1_0)])_L\cup E([\G(1_0)])_S$ and $E([\G(1_1)])_R\cup E([\G(1_1)])_S$}\label{Fig2abra}
\end{figure}
From Figure~\ref{Fig2abra} and the discussion on $E[\G(1_0)]_R$ and $E[\G(1_1)]_L$, we deduce $|E([\G(1_0)])|\ne |E[\G(1_1)]|$.
In particular, $[\Gamma(1_0)]\ncong [\Gamma(1_1)].$
Thus, $A$ has two orbits on $V\G$, that is, $G_0$ and $G_1$. Since $A$ fixes  setwise $G_0$ and $G_1$, $A_{1_0}$ fixes setwise $\Gamma(1_0)\cap G_0$ and $\G(1_0)\cap G_1$ and $A_{1_1}$ fixes  setwise $\Gamma(1_1)\cap G_0$ and $\G(1_1)\cap G_1$.

Now, by looking at  Figure~\ref{Fig2abra}, we deduce that  $A_{1_0}$ fixes $1_1$ and $A_{1_1}$ fixes $1_0$, that is, $A_{1_0}=A_{1_1}$. By checking Figure~\ref{Fig2abra} again, we infer that $A_{1_0}=A_{1_1}$ fixes pointwise
\begin{equation}\label{eq:mine}
1_0,\,\,
(a_1)_0,\,\,
(a_1^{-1})_0,\,\,
(a_2^{-1})_0,\,\,
((a_1a_2)^{-1})_0,\,\,
(a_\kappa^{-1})_0,\,\,
1_1,\,\,
(a_1)_1,\,\,
(a_1^{-1})_1,\,\,
((a_1a_2)^{-1})_1,\,\,
(a_\kappa^{-1})_1.
\end{equation}

Let $B$ be the permutation group induced by $A$ on $G_0$ and let $\Delta$ be the Cayley graph induced by $\G$ on $G_0$. Thus $\Delta:=\Cay(G,R)$. From~\eqref{eq:mine}, $B_1$ fixes $a_1,a_\kappa^{-1}$ and $a_2^{-1}$. Therefore, from Lemma~\ref{Watkins-Nowitz}, $B_1$ fixes pointwise $\langle a_1,a_2,a_\kappa\rangle$. This yields that $A_{1_0}$ fixes pointwise $\langle a_1,a_2,a_\kappa\rangle_0$.

Now we argue by contradiction and we suppose that $A_{1_0}\ne 1$. Assume $A_{1_0}$ fixes pointwise $R_0$. Then, arguing as in the previous paragraph, we obtain that $A_{1_0}$ fixes pointwise $\langle R\rangle_0=G_0$. However, since $A_{1_0}=A_{1_1}$, we deduce that $A_{g_0}=A_{g_1}$ for every $g\in G$, and hence $A_{1_0}$ fixes pointwise $G_1$, contradicting our assumption that $A_{1_0}\ne 1$. Therefore, $A_{1_0}$ does not fix pointwise $R_0$. Let $i$ be the minimal number in $\{1,\ldots,\kappa\}$ such that $A_{1_0}$ does not fix $(a_i)_0$. Since $A_{1_0}$ fixes $(a_1)_0$, $(a_2)_0$ and $(a_\kappa)_0$, we have $i\ge 3$ and $i\ne \kappa$.
Let $\a\in A_{1_0}$ with $((a_i)_0)^\a\ne (a_i)_0$. In particular, $((a_i)_0)^\a$ equals $(a_j)_0$ or $(a_j^{-1})_0$, for some $j\in \{i+1,\ldots,\kappa-1\}$. For the time being, assume that $((a_i)_0)^\a=(a_j)_0$. Since $R$ is a minimal generating set for $G$, there is a unique $4$-cycle in $[G_0]$ passing through $1_0$, $(a_{i-1})_0$ and $(a_i)_0$ and passing through $1_0$, $(a_{i-1})_0$ and $(a_j)_0$. Namely, these two $4$-cycles are $(1_0, (a_{i-1})_0, (a_{i-1}a_i)_0,(a_i)_0)$ and $(1_0, (a_{i-1})_0, (a_{i-1}a_j)_0,(a_j)_0$). Then $(a_{i-1}a_i)_0^\a=(a_{i-1}a_j)_0$.
It follows $$|\G(1_1)\cap \G((a_{i-1}a_i)_0)|=|\G(1_1)\cap \G((a_{i-1}a_j)_0)|,$$ because $1_1^\a=1_1$.
Since $\G(1_1)=L_1\cup S^{-1}_0$ and $\G(1_0)=R_0\cup S_1$, we have
\begin{align*}
\G((a_{i-1}a_i)_0)&=(a_{i-1}a_iR)_0\cup (a_{i-1}a_iS)_1,\\
\G((a_{i-1}a_j)_0)&=(a_{i-1}a_jR)_0\cup (a_{i-1}a_jS)_1.
\end{align*}
 Note that $a_{i-1}a_i\in L$ and $1\in S$. Then $(a_{i-1}a_i)_1\in \G(1_1)\cap \G((a_{i-1}a_i)_0)$, and hence $|\G(1_1)\cap \G((a_{i-1}a_j)_0)|\not=0$. Clearly, $\G(1_1)\cap \G((a_{i-1}a_j)_0)=(S^{-1}\cap a_{i-1}a_jR)_0\cup (L\cap a_{i-1}a_jS)_1$, and it is easy to check (using again the minimality of the generating sets $R$ and $L$) that $|\G(1_1)\cap \G((a_{i-1}a_j)_0)|=0$, a contradiction. An entirely similar argument yields that $(a_i)_0^\a$ cannot be $(a_j^{-1})_0$. From these contradictions, we deduce that $A_{1_0}=1$. Therefore $\G$ is a $2$-GRR.

From our definition of the set $T$, from Lemmas~\ref{lem=prel} and~\ref{lem=prel2} and from Figure~\ref{Fig2abra}, we deduce that $\Theta$ is an $m$-GRR for $G$. Thus $G$ has an $m$-GRR, for every $m\ge 2$.

\smallskip

Suppose next that $G:=\langle a_1,\ldots,a_\kappa,b\rangle$ is a  generalized dicyclic group over the abelian group $H=\lg a_1,\ldots,a_\kappa\rg$. The proof  is similar to the case above and hence we skip some details. Let
\begin{align*}
R:=&\{a_1,a_2,\ldots,a_{\kappa}\} \cup\{a_1,a_2,\ldots,a_{\kappa}\}^{-1},\\
L:=&\{a_1a_2,a_2a_3,\ldots,a_{\kappa-1}a_\kappa\}\cup \{a_1,a_1a_2,a_2a_3,\ldots,a_{\kappa-1}a_\kappa\}^{-1}\cup\{a_\kappa,a_\kappa^{-1}\};\\
S:=&\{1,a_1,a_1a_2,a_\kappa,b\},\\
T:=&R\cup\{a_1a_2,(a_1a_2)^{-1},a_1a_3,(a_1a_3)^{-1}\},\\
x:=&(a_1a_2)^{-1}.
\end{align*}
Observe that $R$, $L$, $T$ and $x$ are as in the abelian case, whereas we have slightly modified $S$.
As before, we have $|R|=|L|=\ell+\kappa$, $|S|=5$, $|T|=\ell+\kappa+4$ and $|R|=|L|=|T|-|S|+1$. Let $\Theta:=\Theta^m(G,R,L,S,T,x)$, $\G:=\BiCay(G,R,L,S)$ and $A:=\Aut(\G)$. We claim that $\G$ is a $2$-GRR over $G$.
Since $|S|=5$, it is easy to obtain $E([\G(1_0)])_L$, $E([\G(1_0)])_S$,
$E([\G(1_1)])_R$ and $E([\G(1_1)])_S$, and we have drawn these edges in Figure~\ref{Fig2abra2}. (In fact, with respect to Figure~\ref{Fig2abra},  only the edge $\{(a_1)_0,1_1\}$ is missing from this graph.)
\begin{figure}[!hhh]
\begin{center}
\begin{tikzpicture}[node distance=1cm,thick,scale=1,every node/.style={transform shape}]
\node[circle,inner sep=1pt,draw, label=-180:$(a_1)_0$](A0){};
\node[right=of A0](AA0){};
\node[right=of AA0, circle,draw, inner sep=1pt, label=right:$1_1$](B0){};
\node[above=of A0, circle,draw, inner sep=1pt, label=left:$(a_2)_0$](A1){};
\node[above=of A1, circle,draw, inner sep=1pt, label=left:$(a_\kappa)_0$](A2){};
\node[above=of B0, circle,draw, inner sep=1pt, label=right:$(a_1a_2)_1$](B1){};
\node[above=of B1, circle,draw, inner sep=1pt, label=right:$(a_\kappa)_1$](B2){};
\node[below=of A0, circle,draw, inner sep=1pt, label=left:$(a_1^{-1})_0$](A3){};
\node[below=of A3, circle,draw, inner sep=1pt, label=left:$(a_\kappa^{-1})_0$](A4){};
\node[below=of B0, circle,draw, inner sep=1pt, label=right:$(a_1)_1$](B3){};
\node[below=of A4, circle,draw, inner sep=1pt, label=left:$(a_2^{-1})_0$](A5){};
\draw (A1) to (B1);
\draw (A2) to (B2);
\draw (B0) to (B1);
\draw [bend left](B0) to (B2);
\draw (A0) to (B3);
\draw (A3) to (B0);
\draw (A4) to (B0);
\draw (A5) to (B3);
\node[below=of A5,circle, inner sep=1pt,label=above:$E({[}\G(1_0){]})_L\cup E({[}\G(1_0){]})_S$](A6){};
\node[right=of B2](BBB2){};
\node[right=of BBB2](BB2){};
\node[right=of BB2, circle,draw, inner sep=1pt, label=left:$(a_1)_0$](AA0){};
\node[right=of AA0](BB){};
\node[right=of BB,circle,draw, inner sep=1pt, label=right:$(a_\kappa)_1$](BB0){};
\node[below=of AA0, circle,draw, inner sep=1pt, label=left:$1_0$](AA1){};

\node[below=of AA2, circle,draw, inner sep=1pt, label=left:$((a_1a_2)^{-1})_0$](AA3){};
\node[below=of AA3, circle,draw, inner sep=1pt, label=left:$(a_\kappa^{-1})_0$](AA4){};
\node[below=of BB0, circle,draw, inner sep=1pt, label=right:$(a_1a_2)_1$](BB1){};
\node[below=of BB1, circle,draw, inner sep=1pt, label=right:$((a_1a_2)^{-1})_1$](BB2){};
\node[below=of BB2, circle,draw, inner sep=1pt, label=right:$(a_\kappa^{-1})_1$](BB3){};
\draw (AA0) to (AA1);

\draw [bend left] (AA1) to (AA4);
\draw (AA1) to (BB0);
\draw (AA1) to (BB1);
\draw (AA3) to (BB2);
\draw (AA4) to (BB3);
\node[below=of BB3,circle, inner sep=1pt,label=below:$E({[}\G(1_1){]})_L\cup E({[}\G(1_1){]})_S$](AA6){};
\end{tikzpicture}
\end{center}
\caption{The subgraphs $E([\G(1_0)])_L\cup E([\G(1_0)])_S$ and $E([\G(1_1)])_R\cup E([\G(1_1)])_S$}\label{Fig2abra2}
\end{figure}

From Figure~\ref{Fig2abra2}, we deduce $|E([\G(1_0)])|\ne |E[\G(1_1)]|$ and hence $[\G(1_0)]\ncong [\G(1_1)]$. Now, following verbatim the proof of the abelian case, we deduce that $A_{1_0}=A_{1_1}$ fixes pointwise $H_0=\langle a_1,\ldots,a_\kappa\rangle_0$ and
$H_1=\langle a_1,\ldots,a_\kappa\rangle_1$. Since $\G(1_0)\setminus H_1=\{b_1\}$, we deduce that
$A_{1_0}$ fixes also $b_1$ and hence $A_{1_0}$ fixes pointwise $\langle H,b\rangle_1=G_1$. Similarly, $A_{1_1}$ fixes pointwise $G_0$ and hence $A_{1_0}=A_{1_1}=1$. Therefore $\G$ is a $2$-GRR.

From our definition of the set $T$, from Lemmas~\ref{lem=prel} and~\ref{lem=prel2} and from Figure~\ref{Fig2abra}, we deduce that $\Theta$ is an $m$-GRR for $G$. Thus $G$ has an $m$-GRR, for every $m\ge 2$.
\end{proof}

\begin{cor}\label{cor=abelian3}Let $G$ be as  in Notation~$\ref{hyp3}$. For every $m\ge 1$, $G$ admits an $m$-$\mathrm{GRR}$ unless
\begin{enumerate}
\item $m=1$ and $G$ is abelian of exponent greater than $2$ or generalized dicyclic,
\item $m=1$ and  $G=\mathbb{Z}_2\times\mathbb{Z}_2\times\mathbb{Z}_2$ or $G=\mathbb{Z}_2\times\mathbb{Z}_2\times\mathbb{Z}_2\times\mathbb{Z}_2$.
\end{enumerate}
\end{cor}
\begin{proof}
Recall that abelian groups of exponent greater than $2$ and generalized dicyclic groups do not admit GRRs: these are the exceptions in part~(1). Therefore, in the light of Lemma~\ref{abe=3}, it suffices to consider the case that $G$ is an elementary abelian group of order at least $8$. From Proposition~\ref{prop=GRR}, $G$ admits a GRR unless $|G|\in \{8,16\}$: these are the exceptions in part~(2).

Suppose then $m\ge 2$; we need to show that $G$ has an $m$-GRR. We argue by induction on $|G|$. We first consider $G:=\mathbb{Z}_2\times\mathbb{Z}_2\times\mathbb{Z}_2$. Here, we rely on a computer-aided computation. Indeed, from Lemma~\ref{lem=prel2}, it suffices to exhibit some subsets $R$, $L$, $S$, $T$ of $G$ with $R=R^{-1}$, $L=L^{-1}$, $T=T^{-1}$ and $|R|=|L|=|T|-|S|+1$, and some $x\in G\setminus S$ satisfying~\eqref{enu2} and~\eqref{enu3} in Lemma~\ref{lem=prel} only with $m=3$. We take $R:=\{a_1,a_2,a_3\}$, $L:=\{a_1,a_1a_2,a_2a_3\}$, $S:=\{1\}$, $T:=\{a_1,a_2,a_1a_2\}$, $x:=a_1$.

Suppose now, $G=\mathbb{Z}_2^\kappa$ with $\kappa\ge 4$. Set $A:=\mathbb{Z}_2^{\kappa-1}$. By induction, for each $m\ge 2$, $A$ admits an $m$-GRR, say $\Delta_m$. From Lemma~\ref{lem=2.2}, either $\Delta_m$ or $\Delta_m^c$ is prime with respect to Cartesian multiplication. Replacing $\Delta_m$ by $\Delta_m^c$, we may suppose that $\Delta_m$ is prime. Consider now $\Theta:=\Delta_m\times \mathbf{K}_2$. From Lemma~\ref{lem=2.1}, we deduce that $\Aut(\Theta)=\Aut(\Delta_m)\times \Aut(\mathbf{K}_2)=A\times \mathbb{Z}_2=G$ and hence $\Theta$ is an $m$-GRR for $G$.
\end{proof}

\subsection{Part 4: Non-abelian exception groups}

\begin{lem}\label{lem=newnew}
Suppose that $G$ is isomorphic to one of the following ten groups:
\begin{align*}
&
D_6,\,\,\,
D_8,\,\,\,
D_{10},\,\,\,
\mathrm{Alt}(4),\,\,\,
\langle a,b,c\mid a^2=b^2=c^2=1,abc=bca=cab\rangle,\,\,\,
\langle a,b\mid a^8=b^2=1,bab=a^5\rangle,\\
&Q_8\times \mathbb{Z}_3,\,\,\,
Q_8\times \mathbb{Z}_4,\,\,\,
\langle a,b,c\mid a^3=b^3=c^2=(ac)^2=(bc)^2=1,ab=ba\rangle,\\
&\langle a,b,c\mid a^3=b^3=c^3=1,ac=ca,bc=cb,a^b=ac\rangle.
\end{align*} Then $G$ admits an $m$-$\mathrm{GRR}$ for every $m\ge 2$.
\end{lem}
\begin{proof}
For this proof we rely entirely on a computer-aided computation: from Lemma~\ref{lem=prel2}, it suffices to exhibit some subsets $R$, $L$, $S$, $T$ of $G$ with $R=R^{-1}$, $L=L^{-1}$, $T=T^{-1}$ and $|R|=|L|=|T|-|S|+1$, and some $x\in G\setminus S$ satisfying~\eqref{enu2} and~\eqref{enu3} in Lemma~\ref{lem=prel} only with $m=3$. We take:
\begin{enumerate}
\item $R:=\{ab,a,a^{-1}\}$, $L:=\{ba,a,a^{-1}\}$, $S:=\{1,ab,b\}$, $T:=\{a,a^{-1},b,ba,ba^{-1}\}$, $x:=a$ when $G=D_6=\langle a,b\mid a^3=b^2=1,a^b=a^{-1}\rangle$;
\item $R:=\{ab,a,a^{-1}\}$, $L:=\{ba,a,a^{-1}\}$, $S:=\{1,ab,b\}$, $T:=\{a,a^{-1},a^2,b,ba\}$, $x:=a$ when $G=D_8=\langle a,b\mid a^4=b^2=1,a^b=a^{-1}\rangle$;
\item $R:=\{ab,a,a^{-1}\}$, $L:=\{ba,a,a^{-1}\}$, $S:=\{1,ab,b\}$, $T:=\{a,a^{-1},a^2,a^{-2},b\}$, $x:=a$ when $D_{10}=\langle a,b\mid a^5=b^2=1,a^b=a^{-1}\rangle$;
\item $R:=\{(2,3,4), (2,4,3), (1,2)(3,4), (1,2,3), (1,3,2)\}$, $L:=\{(2,3,4), (2,4,3), (1,2)(3,4), (1,3,4), (1,4,3)\}$, $S:=\{1\}$, $T:=\{(2,3,4), (2,4,3), (1,2)(3,4), (1,3)(2,4), (1,4)(2,3) \}$, $x:=(2,3,4)$ when $G=\mathrm{Alt}(4)$;
\item $R:=\{a,b\}$, $L:=\{b,c\}$, $S:=\{1,a\}$, $T:=\{a,b,c\}$, $x:=b$ when $G=\langle a,b,c\mid a^2=b^2=c^2=1,abc=bca=cab\rangle$;
\item $R:=\{a,a^{-1},b\}$, $L:=\{a,a^{-1},b\}$, $S:=\{1,a,ab\}$, $T:=\{a,a^{-1},a^2,a^{-2},a^4\}$, $x:=b$ when $G=\langle a,b\mid a^8=b^2=1,bab=a^5\rangle$;
\item $R:=\{a,a^{-1}\}$, $L:=\{b,b^{-1}\}$, $S:=\{1,b,c,ab\}$, $T:=\{a,a^{-1},a^2,c,c^{-1}\}$, $x:=a$ when $G=Q_8\times \mathbb{Z}_3=\langle a,b,c\mid a^4=b^4=c^3=1,b^2=a^2,a^b=a^{-1},ac=ca,bc=cb\rangle$;
\item
$R:=\{a,a^{-1}\}$,
$L:=\{b,b^{-1}\}$,
$S:=\{1,b,c,ab\}$,
$T:=\{a,a^{-1},a^2,c,c^{-1}\}$,
$x:=a$ when $G=Q_8\times \mathbb{Z}_4=\langle a,b,c\mid a^4=b^4=c^4=1,b^2=a^2,a^b=a^{-1},ac=ca,bc=cb\rangle$;
\item
$R:=\{a,a^{-1}\}$,
$L:=\{b,b^{-1}\}$,
$S:=\{1,b,c,ab,ac\}$,
$T:=\{a,a^{-1},b,b^{-1},ab,(ab)^{-1}\}$,
$x:=a$ when $G=\langle a,b,c\mid a^3=b^3=c^2=(ac)^2=(bc)^2=1,ab=ba\rangle$;
\item
$R:=\{a,a^{-1},c,c^{-1}\}$,
$L:=\{b,b^{-1},c,c^{-1}\}$,
$S:=\{1,b,c\}$,
$T:=\{a,a^{-1},b,b^{-1},ac,(ac)^{-1}\}$,
$x:=a$
when $G=\langle a,b,c\mid a^3=b^3=c^3=1,ac=ca,bc=cb,a^b=ac\rangle$.\qedhere
\end{enumerate}
\end{proof}

\subsection{Proof of Theorem~$\ref{theo=main}$.}
\begin{proof}
The proof of Theorem~\ref{theo=main} for non-abelian groups admitting a GRR follows from Corollary~\ref{cor=part1}. The proof of Theorem~\ref{theo=main} for abelian groups and for groups not admitting a GRR follows from Corollaries~\ref{cor=cyclic},~\ref{cor=abelian2} and~\ref{cor=abelian3} and Lemma~\ref{lem=newnew}.
\end{proof}

\section{Bi-digraphical regular representation}
\label{sec5}
From Theorem~\ref{theo=main}, it is easy to deduce a classification for the groups admitting an $m$-DRR.

\begin{proof}[Proof of Theorem~$\ref{theo=Bi-DRR}$]
Suppose that $G$ has no $m$-DRR. When $m=1$, the result follows from~\cite[Theorem~$2.1$]{Babai}. Assume $m\ge 2$. Since every $m$-GRR can viewed as an $m$-DRR by identifying each edge $\{u,v\}$ with the two arcs $(u,v)$ and $(v,u)$, Theorem~\ref{theo=main} implies that $G$ is one of the following groups $Q_8$, $\mz_2^2$, $\mz_n$ with $1\leq n\leq 5$. A direct computation with \texttt{magma} shows that, except for $\mz_1$ and for $\mz_2$, each of these groups admits a $2$-DRR. 

Assume $m=3$. Arguing as above, Theorem~\ref{theo=main} implies that $G$ is either $\mz_1,\mz_2$ or $\mz_3$. These cases can be resolved invoking again \texttt{magma}: $\mz_1$  has no $3$-DRR, but $\mz_2$ and $\mz_3$ both admit a $3$-DRR. For instance, a $3$-DRR for $\mz_2=\langle(1,2)(3,4)(5,6)\rangle$ is  given by the arcs in
$$\{
(1,3),(1,6),(2,4),(2,5),(3,4),(3,6),
(4,3),(4,5),(5,1),(5,2),(6,1),(6,2)
\}.$$

When $m=4$, Theorem~\ref{theo=main} gives that $G$ is either $\mz_1$ or $\mz_2$. Here, \texttt{magma} reveals that $\mz_1$ has no $4$-DRR, but $\mz_2$ has a $4$-DRR.

When $m\ge 5$, Theorem~\ref{theo=main} gives that $G=\mz_1$ and $5\le m\le 9$. Another computation with \texttt{magma} gives that $\mz_1$ has an $m$-DRR when $m\in \{6,7,8,9\}$, but $G$ has no $m$-DRR when $m=5$. For instance, a $6$-DRR for $\mz_1$ (that is, a regular asymmetric digraph) is  given by the arcs in
$$\{(1,6),(1,4),(2,4),(2,5),(3,2),(3,6),(4,3),(4,5),(5,1),(5,3),(6,1),(6,2)\}.\qedhere$$
\end{proof}

\f {\bf Acknowledgement:} The authors were partially supported by the National Natural Science Foundation of China (11571035, 11731002), and the 111 Project of China (B16002).

\end{document}